\documentclass[11pt]{amsart}
\usepackage{amsmath}
\usepackage{enumerate}
\usepackage{float}
\usepackage{comment}
\usepackage{fancyvrb}
\usepackage{caption}
\usepackage{tikz-cd}
\usepackage{subcaption}
\usepackage[margin=1.3 in, includeheadfoot=true]{geometry}
\usepackage{mathtools}
\usepackage{scalerel}
\usepackage{amssymb}
\usepackage{hyperref}
\hypersetup{
    colorlinks=true,
    linkcolor=blue,
    filecolor=magenta,      
    urlcolor=blue,
    citecolor=blue,
}
\usepackage{color}
\usepackage{graphicx}
\usepackage{mathrsfs}

\usepackage{mathabx}
\usepackage{amsthm}
\usepackage{color}

\usepackage[font=small]{caption}
\usetikzlibrary{cd}
\usetikzlibrary{decorations.markings}

\usepackage{url}

\newtheorem*{remark}{Remark}

\def\S{\mathfrak{S}}

\def\Z{\mathbb{Z}}

\def\rothe{\operatorname{Rothe}}

\def\SSYT{\operatorname{SSYT}}
\def\SYT{\operatorname{SYT}}

\def\word{\Psi}
\def\red{\operatorname{red}}

\def\BPD{\operatorname{BPD}}
\def\BPDx{\operatorname{BPD}^\mathsf{x}}
\def\pop{\operatorname{pop}}

\def\id{\operatorname{id}}
\def\EG{\operatorname{EG}}

\def\Des{\operatorname{Des}}
\def\rect{\operatorname{rect}}

\def\r{\includegraphics[scale=0.6]{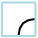}}

\newcommand{\jt}{%
  \begingroup\normalfont
  \includegraphics[height=\fontcharht\font`\B]{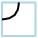}%
  \endgroup
}

\newcommand{\rt}{%
  \begingroup\normalfont
  \includegraphics[height=\fontcharht\font`\B]{rtile.eps}%
  \endgroup
}

\newcommand{\bt}{%
  \begingroup\normalfont
  \includegraphics[height=\fontcharht\font`\B]{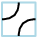}%
  \endgroup
}

\theoremstyle{definition}
\newtheorem{theorem}{Theorem}[section]
\newtheorem{proposition}[theorem]{Proposition}

\newtheorem{lemma}[theorem]{Lemma}
\newtheorem{definition}[theorem]{Definition}
\newtheorem{corollary}[theorem]{Corollary}
\newtheorem{example}[theorem]{Example}

\begin{document}
\title{Schubert Products for Permutations with 
Separated Descents}

% Enter the publication year and the ID number of the paper

% Author name(s)
\author{Daoji Huang}

\address{%
School of Mathematics, University of Minnesota, Minneapolis MN, 55455, USA
}

\begin{abstract}
    We say that two permutations $\pi$ and $\rho$ have \emph{separated descents} at position $k$ if $\pi$ has no descents before position $k$ and $\rho$ has no descents after position $k$. We give a counting formula, in terms of reduced word tableaux, for computing the  structure constants of products of Schubert polynomials indexed by permutations with separated descents, and recognize that these structure constants are certain Edelman-Greene coefficients. Our approach uses generalizations of Sch\"utzenberger's jeu de taquin algorithm and the Edelman-Greene correspondence via bumpless pipe dreams. 
\end{abstract}

\maketitle
\section{Introduction}
Schubert polynomials are representatives
of Schubert classes of the
full flag variety. Products of
Schubert polynomials expand into a 
positive $\Z$-linear combination of
Schubert polynomials, and the coefficients
in this expansion are known as the
\emph{Schubert structure constants}. It
is a famous open problem in algebraic
combinatorics to model the structure
constants nonnegatively using 
combinatorial objects that bypass geometry
in full generality;
or in other words,
to find a \#P-algorithm
that computes the
structure constants.
For a more extensive
introduction of this problem, see
e.g. \cite{lenart2010growth}.

The goal of this paper is to make
some progress on this problem. 
We say that a pair of permutations
$\pi$ and $\rho$ have \textbf{separated
descents} at position $k$ if $\pi$ has no descents
before
position $k$ and $\rho$ has no descents after
position $k$.
This notion was first brought forth
by Knutson and Zinn-Justin \cite{allentalk}, where
they derived a puzzle rule for this 
problem
(that generalizes to equivariant $K$-theory) from studying
quiver varieties. 
This perspective of studying Schubert
and related problems is surveyed in
\cite{KnutsonICM}.
The rule
presented in our work is a tableaux rule
and uses only elementary combinatorial
methods. In
particular, our result generalizes
 results of Kogan \cite{kogan2000schubert},
 Knutson--Yong \cite{knutson2004formula},
 and
Lenart \cite{lenart2010growth}
where different rules were given for
solving the separated descents problem
when one of the permutations
is Grassmannian (i.e., has only a
single descent). One advantage of our approach
is that it naturally recognizes the
separated-descent structure constants
as Edelman-Greene coefficients,
which are the coefficients of the expansion
of Stanley symmetric functions in the
Schur basis.
Our rule is also
natural for computing products of
multiple Schubert polynomials 
indexed by permutations with separated
descents. This gives a generalization
of Purbhoo--Sottile \cite{purbhoo2009littlewood}, where
products of multiple Grassmannian permutations were considered. We note that
even under the condition that every
permutation in the indexing set
is Grassmannian, our rule is still more
general, since it is able to compute the
coefficients of all terms in
the product, whereas 
in \cite{purbhoo2009littlewood} only those 
coefficients of  Grassmannian permutations were considered. 

Our main theorem (Theorem~\ref{thm:main}) is as follows:
\vskip 0.5em

\noindent\textbf{Theorem.}
Suppose $\pi, \rho\in S_n$ such that $\pi$ has
no descents before position $k$ and $\rho$ has
no descents after position $k$. Define
\[\pi\star \rho(i) =\begin{cases}
\pi(i+k)-k & \text{ if }i\in [1-k,0]\\
\rho(i)+n-k & \text{ if }i\in[1,k]\\
\pi(i)-k &\text{ if }i\in[k+1, n]\\
\rho(i-(n-k)) + n-k &
\text{ if }i\in[n+1,2n-k]
\end{cases}\]
where $\pi\star\rho\in S_{[1-k,2n-k]}$.
Let $\sigma\in S_{2n-k}$ such that 
$\ell(\pi\star\rho)-\ell(\sigma)=
\ell((\pi\star\rho)\sigma^{-1})=k(n-k)$.
Let $\lambda_{k\times(n-k)}$ be the partition of
the $k\times(n-k)$ rectangular shape.
 The Schubert structure constant
$c_{\pi,\rho}^\sigma$ is equal to 
the Edelman-Greene coefficient
$j^{(\pi\star\rho)\sigma^{-1}}_{\lambda_{k\times(n-k)}}$, which is
the number of
reduced word tableaux $T$ of shape $\lambda_{k\times(n-k)}$ such
that the permutation given by
the reading word of $T$ (see Section~\ref{sec:eg} for convention) is 
$(\pi\star\rho)\sigma^{-1}$. Furthermore,
$c_{\pi,\rho}^\sigma=0$ for all other $\sigma$ (even though the
number of tableaux may be nonzero).

\vskip 0.5em
We show an example that illustrates the
statement of the theorem.
Let $n=6$, $k=3$, $\pi=135264, \rho=513246$. Then
    $\pi\star\rho=[-2,0,2,8,4,6,-1,3,1,5,7,9]$.
    We have the following Schubert product expansion:
    \[
        \mathfrak{S}_\pi\mathfrak{S}_\rho=\mathfrak{S}_{615243}+\mathfrak{S}_{534162}+\mathfrak{S}_{625134}+\mathfrak{S}_{526143}+2\mathfrak{S}_{624153}
        +\mathfrak{S}_{7152346}+\mathfrak{S}_{7142536}+\mathfrak{S}_{7231546}.
    \]
    For $\sigma=624153$, there are two Coxeter-Knuth classes
    of reduced words for $(\pi\star\rho)\sigma^{-1}$ whose
    reduced word tableaux are of shape $3\times 3$.
    These are:
    \begin{center}
        \includegraphics[scale=0.8]{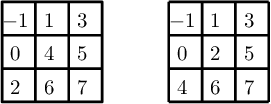}
    \end{center}
        For $\sigma=7142536$, there is one Coxeter-Knuth class
    of reduced words for $(\pi\star\rho)\sigma^{-1}$ whose
    reduced word tableau is of shape $3\times 3$:
    \begin{center}
        \includegraphics[scale=0.8]{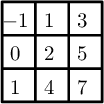}
    \end{center}
    
When $\pi$ and $\rho$ both have a single 
descent at position $k$, our theorem gives a 
combinatorial interpretation for the classical
Littlewood-Richardson coefficients in terms of certain reduced word tableaux of a
rectangular shape, which does not seem to directly match an existing rule in the  literature. However, a version of the theorem
in terms of the generalization of jeu de taquin on semi-standard
Young tableaux can also be stated, and this is given
in Theorem~\ref{thm:jdtversion}.

We take a moment to elaborate on the 
journey that led to the discovery.
The first idea was to
generalize Sch\"utzenberger's 
jeu de taquin algorithm using bumpless
pipe dreams. The key challenge here
was to come up with a bumpless
pipe dream analogue of
skew tableaux. To achieve this,
we introduce a simulation of
jeu de taquin using semi-standard
tableaux, and use the direct bijection
between bumpless pipe dreams for
$k$-Grassmannian permutations and
semi-standard Young tableaux of
the corresponding shape to obtain 
a version of jeu de taquin with
bumpless pipe dreams. The second
challenge was to figure out to what
extent the algorithm can be generalized
beyond Grassmannian permutations. It turns
out that the separated descents condition
fits the bill, and the rectification
process by jeu de taquin
with bumpless pipe dreams for $k$-Grassmannian permutations uses moves
very much like those described
in \cite{LLS}, where a bumpless
pipe dream version of Edelman-Greene
correspondence is given. The full result
is hence achieved by considering
a generalization of the Edelman-Greene
correspondence using bumpless pipe dreams. 
The obstruction to extending our techniques
beyond the separated descent case seems
to be the difficulty of finding a bumpless pipe
dream analogue of skew semi-standard Young
tableaux in direct bijection of the product
of the bumpless pipe dreams for the two input
permutation. The details of this construction
in the separated descent case is introduced
in Section~\ref{subsec:sep1}.

The organization of the paper is as
follows. Section 2 provides background
and sets conventions for the paper.
A combinatorial interpretation of the divided difference
operators is also given in this section.
Section 3 develops the main technical
tools, the rectification and insertion
algorithms on bumpless pipe dreams, and
establishes some necessary properties
of these algorithms.
Section 4 
states and proves our main
result, Theorem \ref{thm:main}, 
as well as the extension
of our result to 
Schubert products for
for multiple
permutations with separated descents.
Section 5 spells out the connection
of our construction to the original
jeu de taquin, and Section 6 ends with
some concluding remarks.

\section{Preliminaries}
We introduce some notations. For $m,n\in\Z$,
$m\le n$ we write $[m,n]:=\{m,m+1,\cdots, n\}$.
Let $S_{[m,n]}$ denote the set of permutations
on the alphabet $[m,n]$. As usual, $S_n$ denotes
the set of permutations on $\{1,\cdots, n\}$.
\subsection{Bumpless pipe dreams and Schubert polynomials}
Bumpless pipe dreams were introduced in
\cite{LLS}. A \textbf{bumpless pipe 
dream} for a permutation $\pi\in S_{[m,n]}$
is a tiling of an $(n-m+1)\times(n-m+1)$ grid
with rows and columns indexed by numbers
in $[m,n]$ by the following six kinds of tiles,
such that $n-m+1$ pipes that travel
from the south border and exit from the east
border are formed. The tile where two pipes
``bump'' is forbidden, and hence the
name ``bumpless pipe dream.''
\begin{center}
    \includegraphics[scale=0.6]{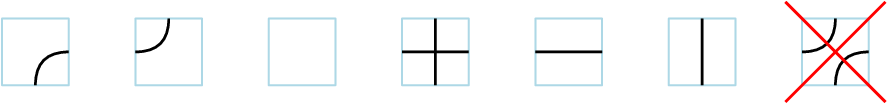}
\end{center}
We also impose the ``reducedness''
condition, that no two pipes are allowed to
cross twice.
The pipes are labelled from $m$ to $n$
on the south border, and reading off the labels
from top to bottom on the east border gives
a permutation in $S_{[m,n]}$. Given a permutation
$\pi\in S_{[m,n]}$, we denote the set of 
bumpless pipe dreams of this permutation
by $\BPD(\pi)$. The tiles in a bumpless pipe
dreams are indexed with
matrix coordinates. The \textbf{Rothe
bumpless pipe dream} $\rothe(\pi)$ is
the unique bumpless pipe dream of 
$\pi$ that does not contain any
\jt-tiles\footnote{These are pronounced ``jay''
as in ``j'', whereas the
\rt-tiles are pronounced ``are''.}, which directly corresponds to the Rothe diagram of the permutation $\pi$. It was shown in
\cite[Section 5.2]{LLS}
any $D\in\BPD(\pi)$ can be obtained
from $\rothe(\pi)$ by performing
a sequence of \textbf{droop moves}, which we briefly recall. A
\textbf{droop} is a local move that swaps an $\rt$-tile of pipe $p$ at position $(a,b)$ with a strictly-southeast blank tile at position $(a+i,b+j)$. Let $R$ be the rectangle with NW corner $(a,b)$ and SE corner $(a+i,b+j)$. After the droop, $p$ travels along row $a+i$ and column $b+j$, $(a,b)$ becomes blank, and $(a+i,b+j)$ becomes a $\jt$-tile. The droop is only allowed if every tile in row $a$ and column $b$ of $R$ contains $p$, $R$ contains only one $\rt$-tile at $(a,b)$, and after the droop we obtain a valid bumpless pipe dream. The reverse operation of a droop is called an 
\textbf{undroop}. We note here that a more general version
of  droops and undroops  will sometimes be used and defined later.

Given $\pi\in S_n$, define the
bumpless pipe dream polynomial
$P_\pi$ by
\[P_\pi(\mathbf{x}):=\sum_{D\in \BPD(\pi)}\operatorname{wt}(D),\]
where $\operatorname{wt}(D):=\prod_{(i,j)\in \operatorname{blank}(D)}x_i$ and $\operatorname{blank}(D)$ is the
set of blank tiles in $D$.

Meanwhile, Schubert polynomials 
were defined by Lascoux and 
Sch\"utzenberger via divded
difference operators $\partial_i$,
defined as
\[\partial_i f:=\frac{f-s_i f}{x_i-x_{i+1}},\]
where $f$ is a polynomial, and
$s_i$ acts by swapping the
variables $x_i$ with $x_{i+1}$.
The \textbf{Schubert polynomial}
$\S_\pi$ for $\pi\in S_n$ is defined
as follows:
\[\S_\pi:=\begin{cases}x_1^{n-1}x_2^{n-2}\cdots x_{n-1} & \text{ if } \pi = n,(n-1),\cdots,2,1\\
\partial_i\S_{\pi s_i} &\text{ if }\pi(i)<\pi(i+1). \end{cases}\]
It was proved in \cite{LLS} and
in \cite{weigandt2021bumpless}
that $P_\pi=\S_\pi$.  We give
another proof by directly computing how divided difference operators act on
bumpless pipe dream polynomials.
The operations defined on
bumpless pipe dreams in the
proof  are analogous to
the ``mitosis'' operations
for pipe dreams. See \cite{knutson2005grobner} and
\cite{MILLER2003223}.
\begin{proposition}
\label{prop:divided}
Let $\pi \in S_n$.
If $\pi(i)<\pi(i+1)$,
then $\partial_i P_\pi=0$.
Otherwise, $\partial_i P_\pi = P_{\pi s_i}.$ 
\end{proposition}

\begin{proof}
Suppose $D\in\BPD(\pi)$. We define a row move within
row $i$ and $i+1$, which is a slight generalization of the  droop move of \cite{LLS} and recalled above, but for simplicity we also call it a droop.  Let $(i,c)$ be the coordinate of an
$\rt$-tile in row $i$, and $p$ the pipe that contains this
$\rt$. If there is a blank tile $(i+1,y)$ in row $i+1$
directly below $p$, the row move can change $p$ so that
$(i,c)$ becomes blank, $(i+1,y)$ becomes $\jt$, and $p$
travels in row $i+1$ instead of $i$ strictly between
columns $c$ and $y$, and any ``kinks'' of other pipes
between columns $c$ and $y$ and blank tiles get shifted up. 

We also define a reverse row move, which for simplicity we also call an undroop. Let $(i+1,y)$ be the coordinate
of a $\jt$-tile in row $i+1$, and let $p$ be the pipe that
contains this $\jt$. If there is a blank tile $(i,c)$ in
row $i$ directly above $p$, the reverse move can change $p$ so
that $(i+1,y)$ becomes blank, $(i,c)$ becomes $\rt$, and
$p$ travels in row $i$ instead of $i+1$ strictly between
columns $c$ and $y$, and any ``kinks'' of other pipes
between columns $c$ and $y$ get shifted down. We also call
this an undroop. In Figure~\ref{fig:rowdroops}, we show
some examples of these row moves and reverse row moves.

\begin{figure}[h]
    \centering
    \includegraphics[scale=0.6]{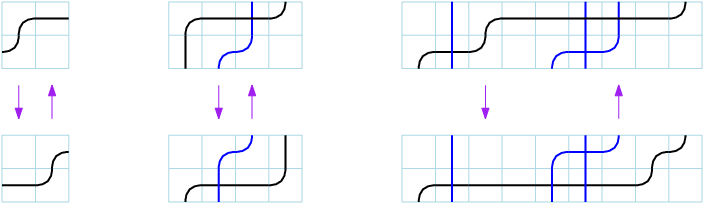}
    \caption{Examples of (generalized) droops and undroops within row $i$ and $i+1$}
    \label{fig:rowdroops}
\end{figure}
We then define an equivalence relation
on $\BPD(\pi)$, where $D_1\sim D_2$
if $D_1$ and $D_2$ are the same
outside of row $i$ and $i+1$, and
are related by a sequence of
droops/undroops as defined above within row
$i$ and $i+1$. 
Let $\mathcal{D}$ be an equivalence
class. 
Consider a bumpless pipe dream $D$ in an equivalence class
$\mathcal{D}$. If $e$ is a blank tile in row $i$ or $i+1$, there are three cases:
\begin{itemize}
    \item[(a)] $e$ is above or below another blank tile,
    \item[(b)] $e$ is above or below a pipe $p$ that exits from row $j$ with $j\le i$ (i.e., $p=\pi(j)$), in which case we say ``$e$ is assigned to $p$'', 
    \item[(c)] $e$ is above the pipe $p=\pi(i+1)$, the pipe that exits from row $i+1$.
\end{itemize}
Let $Q$ denote the set of pipes to which a droop or undroop
is available within row $i$ and $i+1$ of $D$. By definition
of these moves, 
we notice that by restricting to row $i$
and $i+1$, 
%the moves of each pipe $q$ in $Q$ can only
%affect tiles between columns $c_1$ and $c_2$, where $c_1$
%is the column from which $q$ enters row $i+1$, and $c_2$ is
%the column from which $q$ leaves row $i$, or the rightmost
%column of $D$ if $q$ leaves from row $i$. Therefore, 
if a
blank tile is assigned to $q$, it can only be moved by $q$
and not other pipes. If we fix other pipes and only perform droops/undroops on $q$, the set of blank tiles that can be moved by $q$ generate symmetric
polynomial $x_i^\alpha+ x_i^{\alpha-1}x_{i+1}+\cdots + x_{i+1}^\alpha$, where $\alpha$ is the number of such blank
tiles.
Blank tiles of types (a) and (c) are
not affected any droop or undroops.

Let $m$ be the number of blank tiles of type (c). 
We define 
$P_\mathcal{D}:=\sum_{D\in \mathcal{D}}\prod_{(i,j)\in\text{blank}(D)}x_i$,
and notice that 
$P_\mathcal{D}=p(x_1,\cdots, x_{i-1}, x_{i+2},\cdots, x_{n-1})(f(x_i,x_{i+1})+x_i^m)$,
where $p$ is a monomial in variables
not involving $x_i$ and $x_{i+1}$,
$f$ is symmetric in $x_i$ and $x_{i+1}$. (In fact, $f$ can be written as a product of symmetric polynomials in $x_i$ and $x_{i+1}$ indexed by elements in $Q$, times $(x_ix_{i+1})^b$, where $b$ is the number of vertical stacks of two blank tiles in row $i$ and $i+1$.)
We will show that
\[\partial_i P_\mathcal{D}:=\begin{cases}0 & \text{ if } m=0\\
P_\mathcal{D'} &\text{ otherwise}, \end{cases}\]
where $\mathcal{D}'$ is defined in the argument to follow.

If 
$\pi(i) < \pi(i+1)$,
 the pipes $\pi(i)$
and $\pi(i+1)$ do not cross
in any element of $\BPD(\pi)$.
This implies $m=0$ and $P_{\mathcal{D}}$
is symmetric in $x_i$ and $x_{i+1}$,
so $\partial_i P_\mathcal{D}=0$.
Since $\mathcal{D}$ is
arbitrary, $\partial_i P_\pi=0$.

Otherwise $\pi(i)>\pi(i+1)$. Pick $D\in\mathcal{D}$,
 let $y_i$ be the column index
of the last \rt-tile on row $i$
in $D$,
and $y_{i+1}$ be the column index
of the last  \rt-tile on
row $i+1$ in $D$. By assumption,
$y_{i+1}<y_i$.  Again,
if $m=0$ then $\partial_i P_{\mathcal{D}}=0$.
If $m>0$, we construct a (multi-valued) map $\phi$
that sends a $D\in \BPD(\pi)$
to a
$D'\in \BPD(\pi s_i)$ as follows.
Replace the crossing
of pipes $\pi(i)$ and $\pi(i+1)$
with a \bt-tile,
and undroop the NW-elbow
into one of these $m$
boxes in row $i$,
as 
illustrated in Figure~\ref{fig:row move}.
\begin{figure}[h]
    \centering
    \includegraphics[scale=0.6]{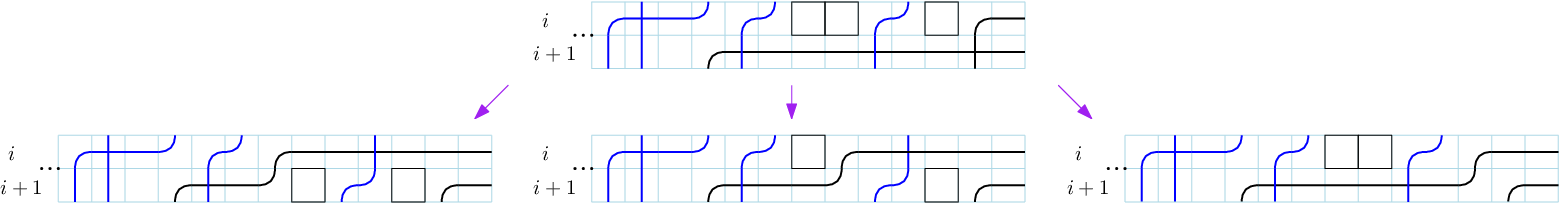}
    \caption{The multi-valued map $\phi$}
    \label{fig:row move}
\end{figure}
We argue the set of diagrams $D'$ that can
be obtained from $D\in\mathcal{D}$, where $\mathcal{D}$ is
an equivalence class in $\BPD(\pi)$ via $\phi$ form an
equivalence class
$\mathcal{D}'$. First notice that if $D_1',D_2'\in\phi(D)$,
then $D_1'\sim D_2'$ by droops/undroops of the pipe
$\pi s_i(i)$. Now suppose $D_1\sim D_2$, then within row
$i$ and $i+1$, $D_1$ and $D_2$ are the same in columns
$\ge y_{i+1}$, and $D_2$ can be obtained from $D_1$ by 
droops/undroops of set of pipes $Q$ with indices
$\pi(j)$ for various $j<i$. If $D_1'\in \phi(D_1)$ and $D_2'
\in \phi(D_2)$, $D_2'$ can be obtained from $D_1'$ by the
same droop/undroop moves of pipes in $Q$, plus possible
droop/undroop moves of the pipe $\pi s_i(i)$. Therefore,
$D_1'\sim D_2'$. For the other direction, define the reverse
operation $\psi$ as follows. For a $D'\in \BPD(\pi s_i)$, 
droop the rightmost $\rt$ in row $i$ into the rightmost
$\rt$ in row $i+1$ so that a $\bt$ is created, and then
replace the $\bt$-tile with a cross tile. Note that
the blank tiles that were assigned to pipe $\pi s_i(i)$ in $D'$
become unmovable in $\psi(D')$, and the other blank
tiles can move the same way as before.
We can then see that if $D_1',D_2'\in \BPD(\pi s_i)$ where
$D_1'\sim D_2'$, then $\psi(D_1')\sim \psi(D_2')$.

Notice that
$\partial_i P_\mathcal{D}=p(x_1,\cdots, x_{i-1},x_{i+2},\cdots, x_{n-1})f(x_i,x_{i+1})(x_i^{m-1}+x_i^{m-2}x_{i+1}+\cdots x_{i+1}^{m-1})= P_{\mathcal{D}'}$. 
We can then conclude that
$\partial_i P_\pi = P_{\pi s_i}$.
\end{proof}

\begin{corollary}[\cite{LLS}]
For every $\pi\in S_n$, $P_\pi=\S_\pi$.
\end{corollary}
\begin{proof}
When $\pi=n,n-1,\cdots, 1$, $\BPD(\pi)$
contains a singleton Rothe
bumpless pipe dream, and 
$P_\pi = x_1^{n-1}x_2^{n-2}\cdots x_{n-1}$.
The inductive case is given by Proposition
\ref{prop:divided}.
\end{proof}

\subsection{Edelman-Greene correspondence}
\label{sec:eg}
We recall some results about the Edelman-Greene correspondence,
and set some conventions for this paper.

\begin{definition}[Coxeter-Knuth insertion] Suppose 
$P$ is a row-and-column strict increasing tableau. Let $x$
be a number to be inserted
into $P$. Initialize the 
row index $i=1$, and follow
the steps below.
\begin{enumerate}
    \item If $x$ is larger than
    or equal to
    all numbers in row $i$ or if row
    $i$ is empty,
    append $x$ at the end
    of the row, and stop. 
    
    \item Otherwise, let $z$ be the leftmost
    number in row $i$ larger than 
    $x$. 
    \begin{itemize}
        \item[(a)] If $z=x+1$ and the value 
        of $x$ is already present in
        row $i$, leave row $i$ unchanged,
        increment $x$ by 1, increment
        $i$ by 1, and go to Step 1. Otherwise
        go to Step 2(b).
        
    \item[(b)] Put $x$ in the position
    of $z$ in $P$ and 
    let $z$ be
    the new value of $x$. Increment
    $i$ by 1,
    then go to Step 1.
    \end{itemize}
\end{enumerate}
\end{definition}
\begin{definition}[Edelman-Greene map] Suppose $\pi$
is a permutation and $\mathbf{i}=i_1i_2\cdots i_\ell$
is a reduced word of $\pi$.
Initialize $P^0$ and $Q^0$ to
both be empty tableaux. For each $j=1,2,\cdots, \ell$, construct 
$P^j$ by inserting the 
number $i_{\ell+1-j}$ into
$P^{j-1}$ using the Coxeter-Knuth
insertion, and $Q_j$ by adding
a new box to
$Q^{j-1}$ so that $P^j$ and 
$Q^j$ have the same shape, and
fill this box with $i$.
Let the \textbf{insertion tableau}
$P(\mathbf{i})$ be $P^\ell$
and  the \textbf{recording tableau}
$Q(\mathbf{i})$ be $Q^\ell$. We also say that
$P(\mathbf{i})$ is the \textbf{reduced word tableau}
of $\mathbf{i}$.
\end{definition}

\begin{remark}
Note that we insert a reduced
word from right
to left into a tableau.
This choice is made for the 
convenience of stating our main
theorem, and is consistent with
the convention used in
\cite{LLS}. For this reason, we will define our
reading order for the insertion
tableau to be row-by-row
from top to bottom, and
right-to-left within each row.
\end{remark}

\begin{definition}
The \textbf{Coxeter-Knuth equivalence}
on the set of reduced words of $\pi$ is
generated by the following elementary
relations:
\begin{itemize}
    \item[(a)] $\cdots jik\cdots \sim \cdots jki\cdots$ if 
    $i<j<k$,
    \item[(b)] $\cdots ikj\cdots \sim \cdots kij\cdots $ if $i<j<k$,
    \item[(c)] $\cdots i(i+1)i\cdots \sim (i+1)i(i+1)$.\qedhere
\end{itemize} 
\end{definition}
Recall that the descent sets of
a standard tableau of shape $\lambda$ is
defined as 
$\Des(S):=\{j: j+1 \text{ appears in a lower row than }j\}$, and the descent set of a reduced word
$\mathbf{i}=i_1\cdots i_l$ is
$\Des(\mathbf{i})=\{j:i_j>i_{j+1}\}$. 
We write the reverse of a reduced word
$\mathbf{i}$ as $\mathbf{i}^{-1}$.
\begin{theorem}[Edelman-Greene correspondence \cite{edelman1987balanced}]
\label{thm:eg}
The map
\[\EG:\mathbf{i}\mapsto (P(\mathbf{i}),Q(\mathbf{i}))\] is 
an injective map from the set of
reduced words for
$\pi$, $\red(\pi)$, to the set of pairs of tableaux
$(P,Q)$ where $P$ is a row-and-column
strict increasing tableau and $Q$ a standard tableau. 
This map has the following properties:
\begin{itemize}
    \item[(a)] For
each $P$, every  standard tableau
$Q$
with the same shape as $P$ appears in the
image.
\item[(b)] Two reduced words of $\pi$ are
Coxeter-Knuth equivalent if and only if
they have the same insertion tableau.
\item[(c)] If $\EG(\mathbf{i})=(P,Q)$
and $Q'$ is another standard tableau,
then $\EG^{-1}(P,Q')$ is Coxeter-Knuth
equivalent to $\mathbf{i}$.
\item[(d)] For each $\mathbf{i}\in\red(\pi)$,
$\Des(\mathbf{i}^{-1})=\Des(Q(\mathbf{i}))$. \qedhere
\end{itemize}
\end{theorem}

We now recall some basic facts about Stanley symmetric fucntions.
For $\pi$ a permutation, the set of reduced compatible
sequences $\operatorname{RCS}(\pi)$ is defined as  $\operatorname{RCS}(\pi):=\{(\mathbf{i},\mathbf{r}):\mathbf{i}=i_1i_2\cdots i_{\ell(\pi)}\in\red(\pi),\mathbf{r}=r_1\cdots r_{\ell(\pi)}, 1\le i_1\le \cdots \le i_{\ell(\pi)}, i_j<i_{j+1}\implies r_j<r_{j+1}\}.$
The \textbf{Stanley symmetric functions} $F_\pi$ 
is defined as
\[F_\pi=\sum_{(\mathbf{i},\mathbf{r})\in\operatorname{RCS}(\pi)} x_{\mathbf{r}},\]
where $x_{\mathbf{r}}=x_{r_1}\cdots x_{r_{\ell(\pi)}}$
for $\mathbf{r}=r_1\cdots r_{\ell(\pi)}$. It is
well-known that the expansion
$F_\pi=\sum_{\lambda}j^\pi_\lambda s_\lambda$ of $F_\pi$ into
the Schur basis has positive coefficients
known as the \textbf{Edelman-Greene coefficients}. The coefficient $j_\lambda^\pi$ has
the combinatorial interpretation
of counting the number of reduced word tableaux of
shape $\lambda$ for $\pi$. This can be proved by
a variant of Theorem~\ref{thm:eg} that bijects 
$\operatorname{RCS}(\pi)$ with the set of pairs of
$(P,Q)$ where $P$ is a row-and-column strict 
increasing tableau whose reading word is a reduced
word of $\pi$, $Q$ a \emph{decreasing} semi-standard
Young tableau of the same shape as $P$. 
($Q$ being decreasing is due to our convention
that inserts a reduced compatible sequence from right to left.)

We also recall some facts about the recording tableaux
for the Edelman-Greene correspondence.
\begin{definition}
Let $Q\in \SYT(\lambda)$. For $1\le i \le |\lambda|-2$,
define the \textbf{elementary dual equivalence}
$h_i$ as an action on $Q$ such that $h_i$ fixes
all $j$ for $j\not\in\{i,i+1,i+2\}$, and if the
reading word of $Q|_{\{i,i+1,i+2\}}$ is of the form
\[xiy\text{ or }x(i+2)y,\] swaps the entries $x$
and $y$ in $Q$, and fixes it otherwise. 
\end{definition}
It is known that the set of standard Young tableaux
are connected by the elementary dual equivalences, see e.g. 
\cite[Proposition 2.14]{haiman1992dual}. Furthermore, dual equivalence 
characterizes the Coxeter-Knuth moves on reduced words, 
see e.g. \cite[Theorem 6.24]{edelman1987balanced},
\cite[Proposition 2.2(a)]{hamaker2014bijective}.

\subsection{Jeu de taquin}
We recall Sch\"utzenberger's
\textbf{jeu de taquin} algorithm on
semi-standard skew tableaux.
Our reference is \cite[Chapter 1]{fulton1997young}.

Let $T$ be a semi-standard skew
tableau of shape $\nu/\mu$. 
The jeu de taquin algorithm on
$T$ rectifies $T$ into a semi-standard tableau by
iterating the following steps.
\begin{itemize}
    \item[(1)] Pick a SE-most empty box 
    in the NW empty region to be 
    ``active''.
    \item[(2)] Repeat this step
    until the current 
    active empty box 
    is on the SE border: pick
    the smaller number from the
    box to the right and the box
    directly below the active
    empty box, and slide that box
    into the active empty box. 
    If there is a tie, pick
    the box below. The newly
    created empty box after sliding
    becomes the new active
    empty box.
\end{itemize}

It is well-known that the final
output is independent of the
choice made in
in Step (1). Given partitions
$\lambda,\mu,\nu$, the Littlewood-Richardson coefficient
$c_{\lambda,\mu}^\nu$ is the number
of skew semi-standard Young
tableaux of shape $\lambda*\mu$ 
that are jeu de taquin equivalent 
to some fixed semi-standard Young tableau of shape $\nu$,
where $\lambda*\nu$ is obtained by placing the
partition diagram of $\mu$ immediately NE of the
diagram of $\lambda$.

\section{Rectification and Insertion}
In this section, we introduce a reverse
insertion algorithm, which we call rectification,
on bumpless pipe dreams with a partition of
marked blank tiles at the NW corner. 
The terminology ``rectification'' comes from
the rectification algorithm jeu de taquin
on skew semi-standard
tableaux, and this connection is explained in
Section \ref{sec:jdtconnection}. We also introduce
an insertion algorithm of a reduced word
into a bumpless pipe dream and when successful,
produces a bumpless pipe dream with a 
partition of marked blank tiles
at the NW corner. Rectification and insertion
are inverses of each other. Our rectification
and insertion generalize those introduced in
\cite{LLS} that realize the
Edelman-Greene correspondence with bumpless
pipe dreams, and will be the main
technical tool to give a combinatorial
interpretation of the separated descent
Schubert structure constants, as well as
relating those to the Edelman-Greene coefficients.

In the rest of the paper, when we say $\sigma\in S_{[a,b]}$
and $F\in\BPD(\sigma)$, we specifically consider
$F$ as a bumpless pipe dream with rows and
columns indexed by $[a,b]$. 
\subsection{Rectification}
\label{subsec:recinsert}
Let $\sigma\in S_{[a,b]}$.
Define $\BPDx(\sigma)$ to be the set
of bumpless pipe dreams of $\sigma$ such that
each tile in a (possibly empty) partition contained inside the
connected region of blank tiles at the NW
corner of $F$ is marked with ``$\mathsf{x}$''.
We require the NW corner of this partition to be
anchored at the NW corner of the bumpless pipe dream.
 Given $F\in \BPDx(\sigma)$,
let $\lambda(F)$ denote the partition formed by tiles
marked with ``$\mathsf{x}$'', and
let $Q\in\SYT(\lambda(F))$. 
We will now describe a rectification
process on $F$ that iteratively removes these
marked tiles 
according to the order specified
by $Q$
and produces a reduced word
of length $|\lambda|$.
During
the entire rectification process, we
also will make sure to maintain the
invariance of the number of 
non-marked blank tiles in each row.

To start, let $\mathbf{i}$ be the empty
reduced word.
Let
$Q'$ be a working copy of $Q$.
In each iteration of the rectification,
find the coordinate $(x,y)$
of $Q'$ with the largest value,
and delete this box from $Q'$.
We then pick the marked blank tile
at $(x,y)$ to be the active one and
perform the following operations.

\begin{itemize}
    \item[(1)] If the current marked tile is not the
    rightmost tile of a contiguous
    block of blank tiles on a row, slide
    the mark to the rightmost tile.
    \item[(2)]
    Suppose the current mark is at position
    $(i,j)$. Let $p$ be the pipe that passes
    through $(i,j+1)$. 
    \begin{itemize}
        \item[(a)] If pipe $p$ contains a 
    $\jt$ in column $j+1$, let $(i',j+1)$ be
    the coordinate of this $\jt$-tile. In other
    words, $i'>i$ is the smallest such that
    $(i',j+1)$ is a $\jt$-tile. First temporarily
    ignore the pipes that cross $p$ in columns
    $j$ and $j+1$ and between row $i$ and $i'$. (Note
    that since no two pipes can cross twice,
    $p$ must be the only pipe that passes through
    $(i',j)$.) Undroop $p$ at $(i',j+1)$ into
    $(i,j)$ so that $(i,j)$ becomes an $\rt$-tile
    and $(i',j+1)$ becomes a blank tile. Move the mark
    at $(i,j)$ to $(i',j+1)$. Now adjust pipes that 
    cross $p$ between row $i$ and $i+1$ so their
    ``kinks shift right''. 
    See Figure~\ref{fig:colmoves}. 
    Go back to step (1).
        \item[(b)] Otherwise, $p$ originates
        from column $j+1$, i.e., $p=j+1$. In this case,
        pipes $j$ and $j+1$ must cross at some
        tile $(i',j+1)$ with $i'>i$. replace this cross
        with a $\bt$-tile, and resolve this bump by undrooping
        the NW-elbow into $(i,j)$,
        and adjust other pipes if necessary in a similar
        fashion as in Step (2a), 
        as shown in Figure \ref{fig:colmove2}. This move
        decrements the number of marked tile by 1 and
        is the end of this iteration of
        rectification. Append $j$ to the end of
        the reduced word $\mathbf{i}$.
        If $F'$ is the marked bumpless pipe dream
        in the beginning of this iteration,
        define $\pop(F', (x,y))$ to be $j$
        and $\nabla(F',(x,y))$ the resulting
        (marked) bumpless pipe dream with one
        fewer mark than $F'$.
    \end{itemize}
     
The process completes when 
$Q'=\emptyset$ and
there
are no marked tiles in the grid.
Denote the output reduced word
$\mathbf{i}$ as $\word(F,Q)$,
and the bumpless pipe dream at
the end of the algorithm $\rect(F,Q)$.
    \begin{figure}[h]
        \centering
        \includegraphics[scale=0.7]{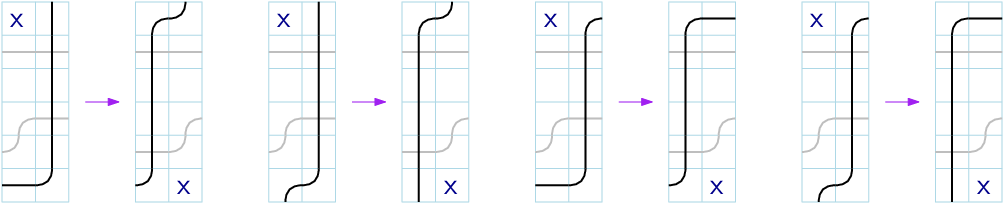}
        \caption{Column moves for rectification in Step (2a)}
        \label{fig:colmoves}
    \end{figure}
    
\begin{figure}[h]

    \centering
    \includegraphics[scale=0.7]{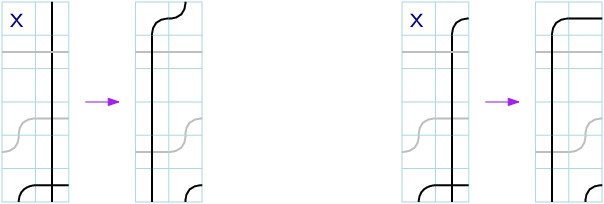}
    \caption{Terminal column move in Step (2b)}
   \label{fig:colmove2}
\end{figure}
\end{itemize}

\begin{example}
Figure~\ref{fig:rect-ex} shows an example of an
iteration of the rectification algorithm. Suppose
the partition of marked tiles 
is anchored at $(1,1)$. This
iteration produces the simple reflection $s_7$.
If $Q\in\SYT((3,2))$ determines that 
the order for removal of the outer boxes of marked
tiles to be $(2,2)$, $(1,3)$, $(2,1)$, $(1,2)$,
$(1,1)$, the resulting reduced word 
$\mathbf{i}=(7,9,6,8,3)$ (or $s_7s_9s_6s_8s_3$ 
in conventional notation).
\end{example}

\begin{figure}[h]
    \centering
    \includegraphics[scale=0.55]{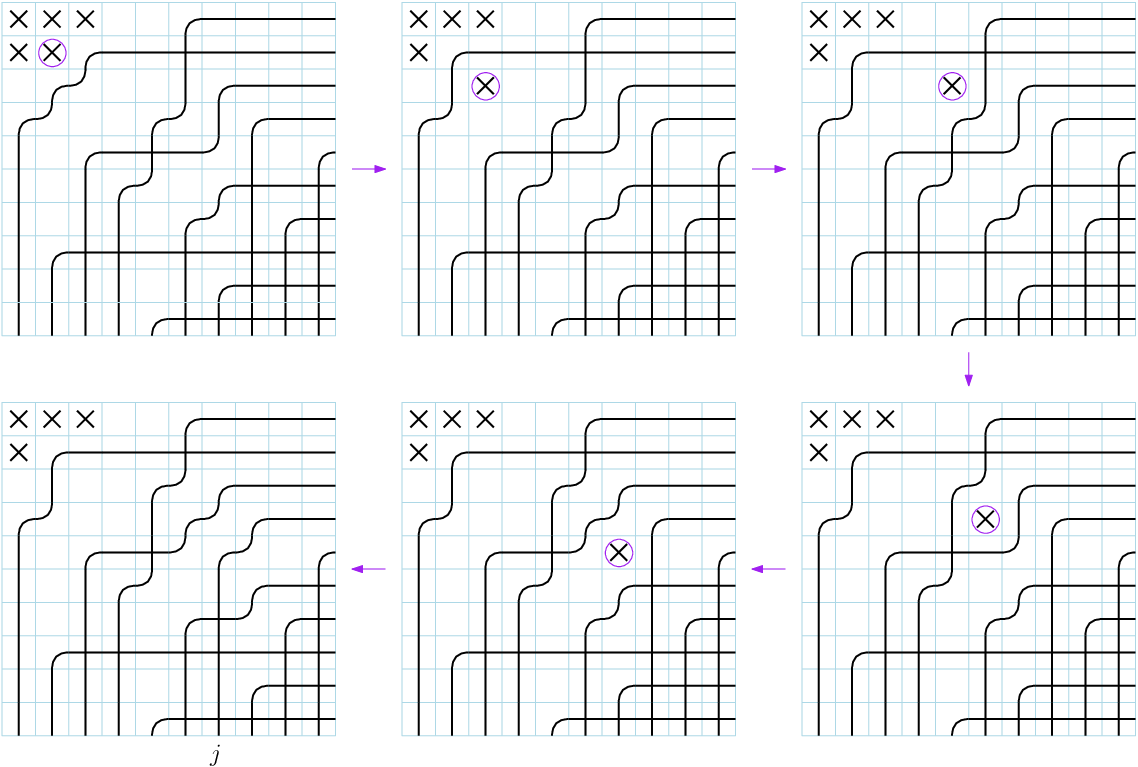}
    \caption{An iteration of rectification}
    \label{fig:rect-ex}
\end{figure}
\begin{remark}
These column moves are slight 
generalizations of 
(the backward direction of)
the column moves
defined in \cite[Section 5.7]{LLS} for a bumpless
pipe dream description of the Edelman-Greene correspondence.
The only difference is that here blank
tiles strictly between rows $i$ and $i'$
are allowed, whereas in \cite{LLS}
they are forbidden. In Step (1) we have
the extra sliding move
that sends 
the marked tile across
a contiguous block of blank tiles
from one end to the other, which
was not needed in their setting. This simple generalization has powerful
consequences. In Section~\ref{sec:jdtconnection},
we relate these moves to jeu de taquin on
semi-standard Young tableaux. 
Because both jeu de taquin and
the Edelman-Greene insertion are realized by
these column moves, we are able to connect
the separated-descent Schubert structure constants
to Edelman-Greene coefficients.
The same
column moves are also used in constructing the
canonical bijection between 
pipe dreams and bumpless pipe dreams \cite{gao2021canonical}.
\end{remark}

\subsection{Insertion}
We now describe an insertion algorithm, which
is the reverse of the rectification algorithm.
Let $F\in\BPDx(\sigma)$ for some 
$\sigma\in S_{[a,b]}$.
We emphasize that we consider the NW corner of $F$
to be position $(a,a)$. 
Let $a\le j < b$. If
pipes $j$ and $j+1$ already cross, the
insertion of $j$ into $F$, denoted $F\leftarrow j$, is not defined. Otherwise, let $(i,j)$
be the location of the $\rt$-tile with $i$
largest, and $(i',j+1)$ the location of the
$\rt$-tile with $i'$ largest. Perform the 
reverse of the terminal column move as 
described in Step (2b) of the rectification
algorithm
(Figure~\ref{fig:colmove2}), which crosses
pipes $j$ and $j+1$ at $(i',j+1)$ and creates
a blank tile at $(i,j)$. Place a mark
at $(i,j)$ and let it be the active one. 
\begin{itemize}
    \item[(1)] If the active mark is not the leftmost
    tile of a contiguous block of unmarked blank tiles
    in a row, slide it to the leftmost one.
    \item[(2)] Suppose the active marked tile
    is at $(i,j)$, and the pipe that passes
    through $(i,j-1)$ is $p$, if it exists. 
    \begin{itemize}
        \item[(a)] Let $(i',j-1)$
    be the coordinate of the $\rt$-tile with
    $i'<i$ largest. Perform the reverse of
    the column move described in Step (2a)
    of the rectification algorithm (Figure~\ref{fig:colmoves}) so that 
    $p$ droops into $(i,j)$, pipes 
    intersecting $p$ between row $i$ and 
    $i+1$ have their ``kinks shift left''
    within columns $j$ and $j+1$, and $(i', j-1)$ becomes a blank tile. Move the mark
    from $(i,j)$ to $(i',j-1)$, and go back to Step (1).
       \item[(b)] If $p$ does not exist and the active mark is connected to
       the existing partition of marked tiles, forming
       a new partition with one more box,
       (or at $(a,a)$ if no tiles were initially marked,) terminate the 
       algorithm with success, and the resulting bumpless pipe dream
       with marked blank tile is denoted $F\leftarrow j$. If the active mark is  disconnected to the existing partition of marked
       tiles, the insertion algorithm fails and
       $F\leftarrow j$ is undefined. 
    \end{itemize}
\end{itemize}

Define the \textbf{insertion footprints} for
$F\leftarrow j$ as the set of
positions the mark appears
at \emph{that later become \jt-tiles}
during the insertion process, plus its
final position. 
We observe that an insertion footprints
set 
consists of tiles that run SW
to NE, with no two tiles on the same
row or column.

\begin{example}
If we reverse the arrows in Figure~\ref{fig:rect-ex},
we get an example of insertion $j$ into the bottom-left
bumpless pipe dream. The insertion footprints consists
of coordinates $(5,7)$, $(4,6)$, $(3,3)$, $(2,2)$, namely the circled positions except for $(3,5)$.
\end{example}
Having defined the insertion of a single simple reflection,
we define the insertion of a reduced word into a 
$F\in\BPDx(\sigma)$.
Let $w\in\S_{[a,b]}$ and $\mathbf{i}=i_1i_2\cdots i_l$ a
reduced word for $w$. Suppose $\ell(w\sigma)=\ell(w)+\ell(\sigma)$.
We define $F\leftarrow \mathbf{i}$ as $(((F\leftarrow i_\ell)\leftarrow i_{\ell-1})\leftarrow\cdots)\leftarrow i_1$,
if every iteration of the insertion is successful.
Otherwise, $F\leftarrow \mathbf{i}$ is undefined.
If $\lambda(F)=\emptyset$, i.e., $F\in\BPD(\sigma)$, the saturated chain of partitions
$\lambda(F\leftarrow i_l)\subset\lambda((F\leftarrow i_l)\leftarrow i_{l-1})\subset\cdots\subset \lambda(F\leftarrow\mathbf{i})$ gives rise to a standard Young tableau of
shape $\lambda(F\leftarrow\mathbf{i})$, which we denote by $Q(F\leftarrow \mathbf{i})$.

When $F$ is the identity bumpless pipe dream, the
insertion of $\mathbf{i}$ into $F$ recovers the
Edelman-Greene correspondence with bumpless pipe
dreams, as described in \cite[Section 5]{LLS}.
The insertion algorithm described above is a generalization
that allows $F$ to be an arbitrary bumpless pipe dream of
an arbitrary permutation $\sigma\in S_{[a,b]}$ with the caveat that the insertion might fail, even when $\sigma$ and the reduced word being inserted are length-additive. However,
this failure can always be remedied by considering $\sigma$
as a permutation in $S_{[a',b]}$
for some $a'< a$ small enough, or in other words,
enlarge $F$ by adding
enough pipes that give the identity
permutation outside the NW corner of
$F$. This is always possible
given $\mathbf{i}$. We refer to this
process as \textbf{back-stabilizing}
$F$. This idea will be useful later, cf. \!Lemma~\ref{lem:backstable}. The idea of
back-stabilizing has showed up in recent work
of Pechenik and Weigandt \cite{PeWei} where 
they gave a positive rule for Schubert products
for inverse Grassmannian permutations, even
though completely different combinatorial
objects are used.

\begin{example}
In Figure~\ref{fig:ins-ex}, the bumpless pipe dream
$F$ is the result of rectification of the top-left
bumpless pipe dream $D$ in Figure~\ref{fig:rect-ex}. Therefore,
the insertion 
$F\leftarrow (7,9,6,8,3)$
 results in $D$. Note that the insertion
$F\leftarrow(6,7,9,6,8,3) = D\leftarrow 6$ is
not defined. However, if we back-stablize $F$
by adding one more pipe numbered 0 and
get $F'$, $F'\leftarrow (6,7,9,6,8,3)$
is defined.
\end{example}
\begin{figure}[h]
    \centering
    \includegraphics[scale=0.55]{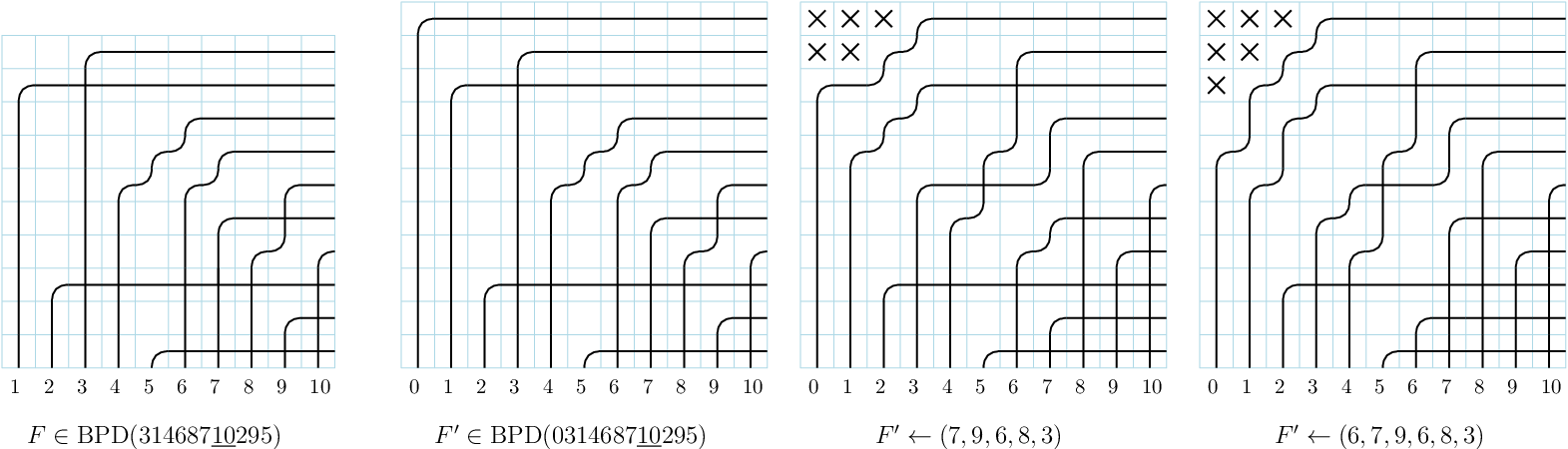}
    \caption{Back-stabilization remedies failures of insertion}
    \label{fig:ins-ex}
\end{figure}

We state a simple lemma that gives a sufficient
condition for one iteration of insertion to 
succeed.
\begin{lemma}
\label{lem:insok}
Suppose $\sigma\in S_{[a,b]}$ and $F\in\BPDx(\sigma)$. Let $\lambda$ denote
the set of marked blank tiles in $F$.
If 
$\lambda$ is not adjacent to any blank tiles,
then $F\leftarrow i$ is defined as long as
pipes $i$ and $i+1$ do not already cross.
\end{lemma}
\begin{proof}
When $i$ and $i+1$ do not cross, 
suppose the insertion of $i$ into $F$ fails
at the final step.
Notice that there are no
pipes weakly NW of this mark; otherwise
the algorithm would not have terminated. 
Because the algorithm
always greedily slides the active mark
to the leftmost tile of a row of contiguous
blank tiles, there must not be any unmarked
blank tiles west of the active mark.
Therefore, the failure is caused by
a vertical strip of unmarked blank tiles 
between the active
mark and $\lambda$. This is not
possible if if $\lambda$ is not adjacent
to any blank tiles, because each column move
always replaces the top-left $\rt$ 
in the bounding rectangle with a mark.
\end{proof}

\subsection{Properties of insertion and rectification}
In this subsection, we discuss properties of
insertion and rectification. 
Corollary~\ref{cor:popdes},
Propositions~\ref{prop:undo}
and \ref{prop:lls} are the main properties of rectification
and insertion required to prove Theorem~\ref{thm:main},
the main theorem. Lemmas~\ref{lem:ij}, \ref{lem:ijrev},
\ref{lem:knuth}, \ref{lem:coxeter} are technical
analysis of the insertion algorithms required to prove
Propositions~\ref{prop:undo} and \ref{prop:lls}. Finally, 
Corollary~\ref{cor:defined} and Lemma~\ref{lem:backstable}
give some necessary conditions under which insertions are defined, which are necessary for the main theorem.
\begin{lemma}
\label{lem:ij}
Let $F\in\BPDx(\sigma)$ and suppose $i<j$. 
When all insertions are defined,
\begin{itemize}
    \item[(1)] The insertion footprints
for $(F\leftarrow i)\leftarrow j$  are
strictly 
to the N/E/NE of the insertion 
footprints for
$F\leftarrow i$. Or in other words,
no insertion footprint of 
$F\leftarrow i$ is N/E/NE
of any insertion footprint of
$(F\leftarrow i)\leftarrow j$.
    \item[(2)] The insertion
    footprints for $(F\leftarrow j)\leftarrow i$ are strictly to the S/W/SW of
    the insertion footprints for 
    $F\leftarrow j$.\qedhere
\end{itemize}
\end{lemma}
\begin{proof}
We show (1); (2) is similar.

Suppose the first tile in the insertion
footprints of $F\leftarrow i$ is $(x_1, y_1)$,
and consider the bumpless pipe dream
$F\leftarrow i$. We know that
$y_1\le i$.
We consider the set of tiles that 
either lies in the ``upper hook'' 
inside the width 2 rectangle
that bounds the column move that
\emph{has happened for
inserting $i$}, 
or in a contiguous block of blank
tiles on a row that bounds the ``sliding''
move of the marked tile. Note that
the insertion footprints are the tiles
at the SW corners of the 
green ``zigzag''
strip illustrated in Figure
\ref{fig:ij}. 

\begin{figure}

    \centering
    \includegraphics[scale=0.6]{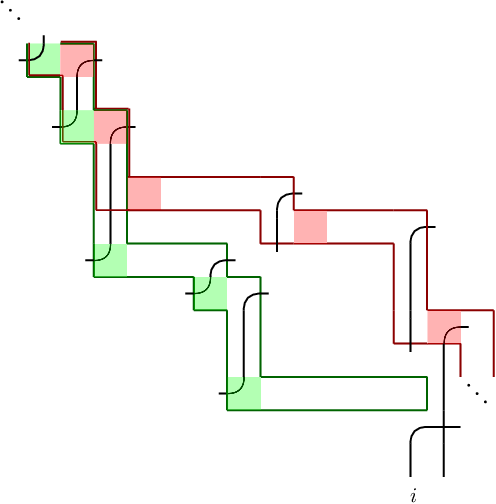}
    \caption{Insertion 
    footprints for $F\leftarrow i$ (green) and $(F\leftarrow i)\leftarrow j$ (red) for $i<j$ in the bumpless pipe dream 
    $F\leftarrow i$}
    \label{fig:ij}
\end{figure}
Notice that the southernmost
possible
coordinate of the lowest
\rt-tile in
column $i+1$ is
$(x_1,i+1)$.
Therefore,
the first tile in the insertion path
of $(F\leftarrow i)\leftarrow j$ with
column index no greater than $i+1$
is at or above row $x_1$. 
If it is at row $x_1$, its coordinate
must be $(x_1',y_1')=(x_1,y_1+1)$. 
In this case, the insertion footprints
for $(F\leftarrow i)\leftarrow j$
are exactly the set of tiles immediately
to the right of each tile in the
insertion footprints of $(F\leftarrow i)$.
Otherwise,
it is on
$(x_1', y_1')$ with
$x_1'<x_1$.
We consider the set of tiles 
that lie in the ``lower hook'' 
inside the width 2 rectangles that
bound column moves that are 
\emph{about to happen for inserting
$j$}, union the contiguous blocks of
blank tiles on a row that bound
the sliding moves of the marked tile.
The insertion footprints of 
$(F\leftarrow i)\leftarrow j$
are the tiles immediately to the 
right of the SW corners of this
zigzag strip. Now imagine travelling
from $(x_1', y_1')$ along this region.
Since going west horizontally
allows only the use of blank tiles,
we may never cross past the green strip.
In the case when two
red and green
vertical segments coincide, we
end up in the same situation as the
first case. Thus our claim
about the insertion footprints
follows.
\end{proof}
Running the proof of Lemma~\ref{lem:ij} backwards,
we have the analogous statements for rectification.
\begin{lemma}
\label{lem:ijrev}
Let $F\in\BPDx(\sigma)$, and $b_1:=(x_1,y_1)$,
$b_2:=(x_2,y_2)$ be two outerboxes of the partition
formed by marked blank tiles of $F$, such that
$x_1>x_2$. Namely, $b_1$ lies SW of $b_2$. Let $F_1:=\nabla(F, b_1)$ 
and $F_2:=\nabla(F,b_2)$. Then
\begin{itemize}
    \item[(1)] $\pop(F, b_1) < \pop(F_1,b_2)$
    \item[(2)] $\pop(F, b_2) > \pop(F_2,b_1)$.\qedhere
\end{itemize}
\end{lemma}

% The 
% following is immediate by Lemma \ref{lem:ij}.
% \begin{corollary}
% \label{cor:des}
% Let $F\in\BPD(\sigma)\subset \BPDx(\sigma)$. 
% Then when $F\leftarrow\mathbf{i}$ is defined,
% $\Des(Q(F\leftarrow \mathbf{i}))=\Des(\mathbf{i}^{-1})$.
% \end{corollary}
Recall that $\word(F,Q)$ denotes the reduced
word obtained by the rectification algorithm on 
$F$ in the order specified by $Q$. The following
Corollary follows from Lemma~\ref{lem:ijrev}.
\begin{corollary}
\label{cor:popdes}
Let $F\in\BPDx(\sigma)$ and $Q\in \SYT(\lambda(F))$,
and let $\mathbf{i}:=\word(F,Q)$
Then $\Des(\mathbf{i}^{-1})=\Des(Q)$. 
\end{corollary}
\begin{lemma}
\label{lem:knuth}
Let $F\in\BPDx(\sigma)$ and suppose $i<j<k$. When all the insertions
are defined,
\begin{itemize}
    \item[(1)] $((F\leftarrow j)\leftarrow i)\leftarrow k = ((F\leftarrow j)\leftarrow k)\leftarrow i$
    \item[(2)] $((F\leftarrow i)\leftarrow k)\leftarrow j = ((F\leftarrow k)\leftarrow i)\leftarrow j$. \qedhere
\end{itemize}
\end{lemma}
\begin{proof}
The proof  is similar to 
the proof of \cite[Lemma 5.25]{LLS}.
In particular, we supply the details
omitted in part (2) of their proof.

To show (1), we notice that since the
insertion footprints of 
$(F\leftarrow j)\leftarrow i$ lie
SW of the insertion footprints of
$(F\leftarrow j)$ by Lemma \ref{lem:ij},
the area NE of the insertion 
footprints is not affected by the
insertion of $i$. Similarly, the area
SW of the insertion footprints $(F\leftarrow j)$
is not affected by the insertion 
$(F\leftarrow j)\leftarrow k$. Therefore,
after the insertion $(F\leftarrow j)$,
the insertions of $i$ and $k$ commute.

To show (2), first we consider the case
when  there is no tile $(x,y)$ in
the insertion footprints of $F\leftarrow i$
such that $(x,y+1)$ is in the insertion
footprints of $(F\leftarrow i)\leftarrow k$. In this case, the insertion footprints
of $(F\leftarrow i)$ still are \jt-tiles
in the bumpless pipe dream
$(F\leftarrow i)\leftarrow k$. This means
that we can undo the insertion of $i$
in $(F\leftarrow i)\leftarrow k$ and
get $(F\leftarrow k)$. In other words,
in this case $(F\leftarrow i)\leftarrow k=(F\leftarrow k)\leftarrow i$.

Otherwise, let $(x,y)$ be the first
tile in the insertion footprints of
$F\leftarrow i$ such that 
$(x, y+1)$ is in the insertion footprints
of $(F\leftarrow i)\leftarrow k$.
This means that the insertion of
$i$ and $k$ \emph{before either path reaches
$(x,y)$} commute, and therefore
the footprints of 
$((F\leftarrow i)\leftarrow k)\leftarrow j$ are the same as $((F\leftarrow k)\leftarrow i)\leftarrow j$ in this
initial segment and is NE of the path
for $i$ and SW of the path for $j$. 

Now consider the situation where we
insert $i$ until the marked tile reaches
$(x_i, y+1)$ for some $x_i>x$, and then
insert $k$ until the marked tile reaches
the furthest tile $(x_k, y_k)$ where $x<x_k<x_i$
and $y_k>y+1$.
(Note that $(x_k, y_k)$ might be the
rightmost tile of some contiguous block
of blank tiles, in which case it will
not become a footprint.
It's easy to check
given our conditions these tiles must
exist.) Then, we insert $j$ until it
reaches the furthest tile $(x_j,y_j)$
with $y_j>y+1$ and $x_j>x_k$. One can
check that it must be the case that
$x_j \le x_i$ and $y_j\le y_k$.

We now perform three droops moves
that advance the marked tiles for
$i,j,k$ in that order. Notice that
all three moves belong to the same
active pipe, and the results are the
same as if the moves were done in the
order $k,i,j$, as illustrated in Figure~\ref{fig:ikjequalskij}.
\begin{figure}
    \centering
    \includegraphics[scale=0.7]{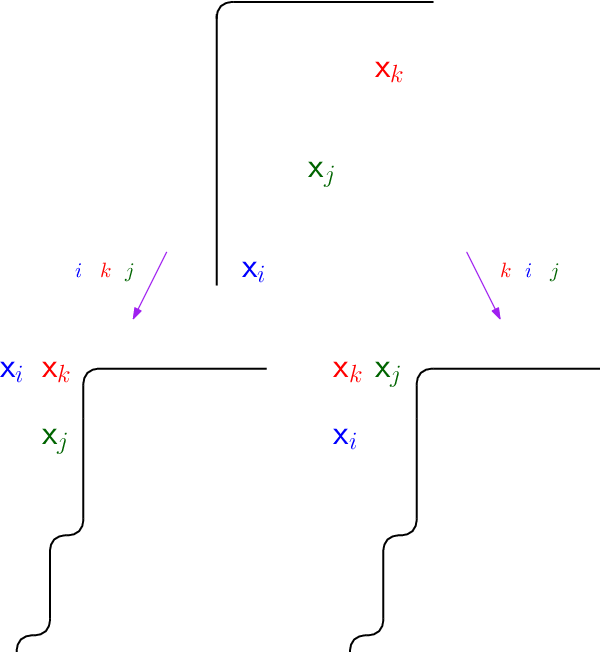}
    \caption{Advancing marks for
    $i$, $k$, $j$ vs. advancing
    marks for $k$, $i$, $j$. Blue
    represents $i$, red represents
    $k$, and green represents $j$.
    The \r-tile in the top 
    diagram has coordinate
    $(x,y)$.}
    \label{fig:ikjequalskij}
\end{figure}
From the diagram on the left, we continue
the insertion process by completing
the insertion of $i$, and then $k$,
and finally $j$. From the diagram on the
right, we do so in the order $k,i,j$. 
By similar reasoning as in part (1), 
we see that these result in the same
diagram when finished.
\end{proof}

\begin{lemma}
\label{lem:coxeter}
Let $F\in\BPDx(\sigma)$. When all the insertions are defined,
$((F\leftarrow i)\leftarrow i+1)\leftarrow i=((F\leftarrow i+1)\leftarrow i)\leftarrow i+1$.
\end{lemma}
\begin{proof}
The same reasoning in
the proof of \cite[Lemma 5.26]{LLS} applies here.
\end{proof}

\begin{proposition}
\label{prop:undo}
Let $P$ be a bumpless pipe dream with a region of
marked blank tiles of shape $\lambda$
at its NW corner,
and let 
$Q$ be a standard tableau of shape $\lambda$. 
Then $C:=\{\word(P,S): S\in \SYT(\lambda)\}$ forms
a single Coxeter-Knuth equivalence class.
Furthermore, the result of rectification
$\rect(P,Q)$ is independent of $Q$. 
\end{proposition}
\begin{proof}
By running the proof of Lemmas 
\ref{lem:knuth} and \ref{lem:coxeter}
backwards, we may show that if $Q,Q'\in\SYT(\lambda)$
are dual equivalent, then $\rect(P,Q)$ and $\rect(P,Q')$
are the same, and $\Psi(P,Q)$ and $\Psi(P,Q')$
differ by a single Coxeter-Knuth move. We omit
the details. Therefore by the fact that
dual equivalences connect all $\SYT(\lambda)$,
the words in $C$ are all
Coxeter-Knuth equivalent, and $\rect(P,Q)$
is independent of $Q$. $C$ is a single 
Coxeter-Knuth class because it is in bijection
with $\SYT(\lambda)$.
\end{proof}
Because rectification is independent of
order, we write $\rect(P):=\rect(P,Q)$
for any $Q\in\SYT(\lambda(P))$.

The following Corollary is immediate from
Proposition~\ref{prop:undo}.
\begin{corollary}
\label{cor:defined}
Let $F$ be a bumpless pipe dream and 
$\mathbf{i}$ a reduced word. If $F\leftarrow\mathbf{i}$ is defined,
then for any $\mathbf{i}'$ Coxeter-Knuth
equivalent to $\mathbf{i}$, $F\leftarrow{\mathbf{i}'}$
is also defined.
\end{corollary}

The following proposition generalizes 
 \cite[Theorem 5.19]{LLS}.
 
\begin{proposition}
\label{prop:lls}
Let $\sigma,w\in S_{[a,b]}$ such that
$\ell(w\sigma)=\ell(w)+\ell(\sigma)$. 
Suppose $\mathbf{i}$ is a reduced word
of $w$ and $F\in\BPD(\sigma)$. 
Denote the shape of the reduced word
tableau of $\mathbf{i}$ by $\lambda(\mathbf{i})$.
The following statements hold.
\begin{itemize}
    \item [(1)] For a fixed $P\in\BPDx(w\sigma)$
    with $\ell(w)$ marked tiles,
    the set
    $C_P=\{\mathbf{j}\in \red(w): F\leftarrow \mathbf{j}=P\}$,
    when non-empty,
    is a single Coxeter-Knuth equivalence class.
    \item [(2)] If $F\leftarrow \mathbf{i}$ is defined, $\lambda(F\leftarrow\mathbf{i})=\lambda(\mathbf{i})$.  \qedhere
\end{itemize} 
\end{proposition}

\begin{proof}
For (1), that $C_P$ is a union of Coxeter-Knuth
equivalence classes
follows from Lemmas \ref{lem:knuth},
\ref{lem:coxeter}, and Corollary~\ref{cor:defined}.
It forms a single Coxeter-Knuth class
follows from Proposition~\ref{prop:undo}. 

For (2), suppose $F':=F\leftarrow\mathbf{i}$ is defined.
Let $C(\mathbf{i})$ be the Coxeter-Knuth equivalence
class of $\mathbf{i}$. 
By Proposition~\ref{prop:undo}, 
$C(\mathbf{i})=\{\word(F',S):S\in\SYT(\lambda(F'))$.

By Corollary~\ref{cor:popdes}, 
\[\{\Des(S):S\in\SYT(\lambda(F'))\}=\{\Des(\mathbf{i}^{-1}):\mathbf{i}\in C(\mathbf{i})\}.\]
If $C$ is Coxeter-Knuth equivalence class and 
$\lambda_C$ the shape of the reduced words
tableaux of words in $C$,
The equality of multisets
$\{\Des(\mathbf{j}^{-1}):\mathbf{j}\in C\}=\{\Des(S):S\in\operatorname{SYT}(\lambda_C)\}$
is a property of Edelman-Greene correspondence.
Also recall from the theory of fundamental quasi-symmetric
functions that
the multiset $\{\Des(S):S\in\operatorname{SYT}(\lambda)\}$
uniquely determines $\lambda$. 
It follows that $\lambda(F')=\lambda(\mathbf{i})$.
\end{proof}

\begin{lemma}
\label{lem:backstable}
Let $\sigma\in S_b$, $\mathbf{i}$ a reduced word
of $w\in S_{[a,b]}$ where $a\le 0 < b$. Assume
$\ell(w\sigma)=\ell(w)+\ell(\sigma)$. Let $P(\mathbf{i})$ be
the reduced word tableau for $\mathbf{i}$.
If the number of rows of $P(\mathbf{i})$ is no greater than
$1-a$, then for any $F\in\BPD(\id_{[a,0]}\oplus\sigma)$ where 
$\id_{[a,0]}\oplus\sigma\in S_{[a,b]}$, 
$F\leftarrow \mathbf{i}$ is defined. 
\end{lemma}
\begin{proof}
Recall that we read a word from right to left when
performing insertion, which is the same as inserting
$\mathbf{i}^{-1}$ from left to right.
By Corollary~\ref{cor:defined}, we may assume without
loss of generality that the recording
tableau $Q(\mathbf{i})$ reads $1,\cdots,\ell(\mathbf{i})$
when read row-by-row from top to bottom, and within
each row left to right. Suppose $Q(\mathbf{i})$
has $m$ entries in the last row.
By Lemma~\ref{lem:insok}, since there are no 
unmarked blank tiles in the non-positive rows,
the first $\ell(\mathbf{i})-m$
iterations of insertion of $\mathbf{i}$ must
be successful. Since $\Des(\mathbf{i}^{-1})=\Des(Q(i))$
the last $m$ letters of $\mathbf{i}^{-1}$ must be
increasing. By the proof of Lemma~\ref{lem:ij} (1) as
well as the fact that there are no unmarked
blank tiles in non-positive rows, we may see that
the insertion of the remaining letters must also
be successful, by similar reasoning as in Lemma~\ref{lem:insok}.
\end{proof}

\section{Schubert products for permutations
with separated descents}

In this section, we first introduce a construction 
that generalizes the construction of placing
a semi-standard Young tableau NE of another and
creating a skew tableau which can then be 
rectified, using bumpless pipe dreams with marked
blank tiles as introduced in the previous section.
The explicit connection will be explained in
Section~\ref{sec:jdtconnection}. We then prove
our main theorem in Section~\ref{subsec:main}
using the properties we developed for insertion
and rectification in Section~\ref{subsec:recinsert}.
\label{sec:sep}

\subsection{The star operation for permutations with separated descents}
\label{subsec:sep1}
We say that a permutation 
$\pi\in S_n$ has a \textbf{descent}
at position $i$ if $\pi(i)>\pi(i+1)$. We
denote the set of $\pi$'s descents 
$\Des(\pi):=\{i\in\{1,2,\cdots,n-1\}:\pi(i)>\pi(i+1)\}$.

Suppose $\pi,\rho\in S_n$ have \textbf{separated descents}
at position $k$; namely for all
$i\in\Des(\pi)$, $i\ge k$  and for all 
$i\in\Des(\rho)$, $i\le k$ . 
We define an operation $\star$ on permutations
$\pi$ and $\rho$, where $\pi\star\rho\in S_{[1-k,2n-k]}$, as follows:

\[\pi\star \rho(i) =\begin{cases}
\pi(i+k)-k & \text{ if }i\in [1-k,0]\\
\rho(i)+n-k & \text{ if }i\in[1,k]\\
\pi(i)-k &\text{ if }i\in[k+1, n]\\
\rho(i-(n-k)) + n-k &
\text{ if }i\in[n+1,2n-k].
\end{cases}\]
We emphasize that the symmetric group $S_{[1-k,2n-k]}$ is on $2n$ numbers that involve non-positive numbers.

Now let $D\in\BPD(\pi)$ and $E\in \BPD(\rho)$. We define a procedure to produce a 
bumpless pipe dream with marked blank tiles $D\star E$
in the $2n\times 2n$ grid with indices $1-k,\cdots, 2n-k$
by the following construction.

We separate $D$ into two blocks, $D_\text{top}$ and $D_\text{bot}$,
where $D_\text{top}$ contains the first $k$ rows of $D$ and $D_\text{bot}$ contains the last $n-k$ rows of $D$.
Also separate $E$ into $E_\text{top}$ and $E_\text{bot}$ in a similar
fashion. Pictorially, $D=\frac{D_\text{top}}{D_\text{bot}}$ and $E=\frac{E_\text{top}}{E_\text{bot}}$. Since the first $k$ values of
$\pi$ are in increasing order, properties of bumpless pipe dreams
ensure that the NE corner of $D$ looks like
\begin{center}
\includegraphics[scale=0.55]{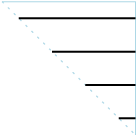}
\end{center}
with a total of $k$ rows. Call this $\nabla_k$. We now build
$D\star E\in\BPDx(\pi\star\rho)$ according to the 
the schema shown in Figure~\ref{fig:schema}.
\begin{figure}[h]
  \centering
     \includegraphics[scale=0.55]{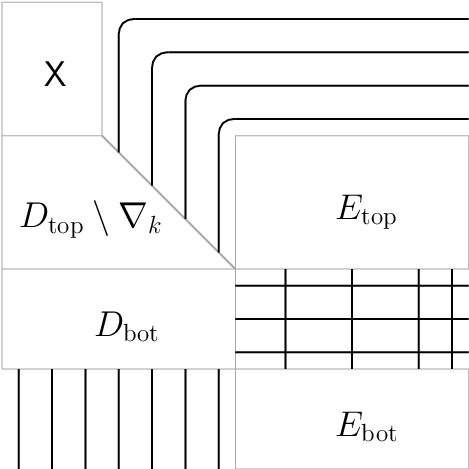} 
    \caption{Schema for $D\star E$}
    \label{fig:schema}
\end{figure}

Notice that there are $k$ vertical pipes in the
region between $E_\text{top}$ and $E_\text{bot}$,
connecting the two parts and making the diagram
a valid bumpless pipe dream. The region marked
with $\mathsf{X}$ in the NW corner are blank tiles
of the $k\times(n-k)$ rectangular shape, each
marked with an ``$\mathsf{x}$''.

    \begin{figure}[h]
    \centering
     \raisebox{0.35\height}{\includegraphics[scale=0.7]{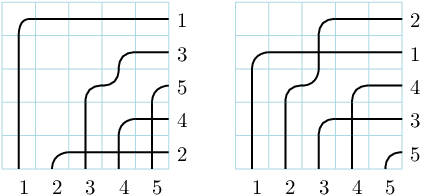}}\hskip 1cm \includegraphics[scale=0.7]{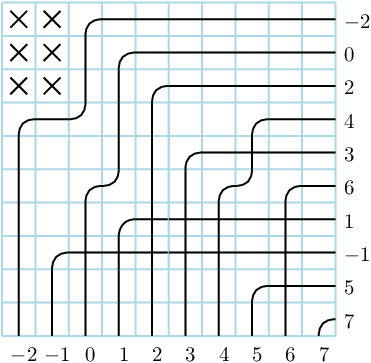}
     \caption{$\pi=13542$, $\rho=21435$, $n=5$, $k=3$, $D\in\BPD(\pi)$, $E\in\BPD(\rho)$, $D\star E\in \BPD^\boxtimes(\pi\star\rho)$}
     \label{fig:DstarE}
    \end{figure}

Let $\BPD^\boxtimes(\pi\star \rho)\subset \BPDx(\pi\star\rho)$
denote the set of bumpless pipe dreams
with marked blank tiles
of $\pi\star\rho$ such that
the first $k$ rows are of the form
that the initial $k$ columns consist of 
marked blank tiles,  followed by
$k$ non-intersecting
hook-shaped pipes that completely
occupy
the
$k\times(n+k)$ remaining columns. This looks
like the first $k$ rows in Figure~\ref{fig:schema}.

By construction, it is easy to see that
$D\star E\in \BPD^\boxtimes(\pi\star\rho)$.
For a concrete example, see Figure \ref{fig:DstarE}.

\begin{lemma}
\label{lem:bijection}
There is a direct bijection between
$\BPD^\boxtimes(\pi\star \rho)$
and
 $\BPD(\pi)\times \BPD(\rho)$. In particular,
if $F\in \BPD^\boxtimes(\pi\star \rho)$ 
corresponds to $(D,E)\in \BPD(\pi)\times \BPD(\rho)$
under the bijection, then for each $i$
the number of unmarked blank tiles in row $i$ of
$F$ is equal to the total number of blank
tiles in row $i$ of $D$ and $E$.
\end{lemma}
\begin{proof}
Suppose $F\in\BPD^\boxtimes(\pi\star \rho)$.
By the definition of $\pi\star\rho$,
the last descent of $\pi\star\rho$ is less than $n$,
and therefore
the rows $[n+1,2n-k]$ of $F$ do not contain
any blank tiles and are completely
determined by the rows above.
First consider the $(n-k)\times n$
region of $F$ with rows $[k+1,n]$
and columns $[1-k,n-k]$.
The pipes $\pi\star\rho(i)=\pi(i)-k\le n-k$
for
$i\in[k+1,n]$. This means that 
these $n-k$ pipes travel from
the south edge of the region
and exit from
the east edge of this region.
The pipes $\pi(i+k)-k\le n-k$ for $i\in[1-k,0]$ travel from the south
edge and exit from
the north edge of this region.
Their configuration agrees with the
last $(n-k)$ rows of
 a bumpless pipe dream of
$\pi$.

Now consider the $k\times n$
region of $F$ with rows
$[1,k]$ and columns $[1-k,n-k]$.
By the definition of 
$\pi\star\rho$, 
since $\rho(i)+n-k>n$,
the only pipes
in this region have indices
$\pi(i+k)-k\le n-k$ for $i\in[1-k,0]$,
and all pipes with indices $\pi(i+k)-k$
for $i\in[1-k,0]$ enter this region
from the south edge and exit from 
columns $[n-k+1,n]$ of the
north edge without ever crossing
each other. If we reroute these pipes
by modifying the tiles in the 
triangular region (weakly) above the
$k$th diagonal counting from the NE
corner of this region,
so that the pipes
exit from the east edge instead,
we can make this region agree with
the first $k$ rows of a bumpless
pipe dream of $\pi$, which can be
combined with the bottom region
to make up a bumpless pipe dream of 
$\pi$.

Consider the $k\times n$ region of
$F$ with rows $[1,k]$ and columns
$[n-k+1,2n-k]$. By the definition
of $\pi\star\rho$ and the discussion above,
the only pipes that show up in this
region are $\rho(i)+n-k>n-k$, and these
pipes enter from the south edge
of this region and exit from the east.
Therefore, this region agrees with
the first $k$ rows of a bumpless pipe
dream of $\rho$.
Furthermore, there are no blank tiles
in the $(n-k)\times n$ region with rows
$[k+1,n]$ and columns $[n-k+1,2n-k]$. Since
$\rho$ has no descents after $k$, 
its bumpless pipe dream is completely
determined by the first $k$ rows. This
gives us a bumpless pipe dream of $\rho$.

It is easy to see that the inverse of
this process is the procedure of constructing $D\star E$ as described above.
We have established the desired bijection.
\end{proof}
\begin{remark}
Lemma~\ref{lem:bijection} is the reason why
we need the ``separated descent'' assumption. If
this condition is not satisfied, it seems very
difficult to construct a set like $\BPD^\boxtimes(\pi\star\rho)$ in weight-preserving bijection with
$\BPD(\pi)\times \BPD(\rho)$.
\end{remark}

\subsection{Main theorem}
\label{subsec:main}
\begin{theorem}
\label{thm:main}
Suppose $\pi, \rho\in S_n$ such that $\pi$ has
no descents before position $k$ and $\rho$ has
no descents after position $k$.
Let $\sigma\in S_{2n-k}$ such that 
$\ell(\pi\star\rho)-\ell(\sigma)=
\ell((\pi\star\rho)\sigma^{-1})=k(n-k)$.
Let $\lambda_{k\times(n-k)}$ be the partition of
the $k\times(n-k)$ rectangular shape.
 The Schubert structure constant
$c_{\pi,\rho}^\sigma$ is equal to 
the Edelman-Greene coefficient
$j^{(\pi\star\rho)\sigma^{-1}}_{\lambda_{k\times(n-k)}}$, which is
the number of
reduced word tableaux $T$ of shape $\lambda_{k\times(n-k)}$ such
that the permutation given by
the reading word of $T$  is 
$(\pi\star\rho)\sigma^{-1}$. Furthermore,
$c_{\pi,\rho}^\sigma=0$ for all other $\sigma$ (even though the
number of tableaux may be nonzero). 
\end{theorem}
\begin{proof}
 We will prove the theorem by constructing
a bijection between $\BPD(\pi)\times\BPD(\rho)$
with \[\bigsqcup_{\sigma\in \Sigma} \BPD(\sigma)\times \{C(\mathbf{i}): \mathbf{i}\in\red((\pi\star\rho)\sigma^{-1}), \operatorname{shape}(P(\mathbf{i}))=\lambda_{k\times(n-k)}\},\] where $\Sigma=\{\sigma\in S_{2n-k}:\ell(\pi\star\rho)-\ell(\sigma)=
\ell((\pi\star\rho)\sigma^{-1})=k(n-k)\}$
(in other words, $\sigma$ must be below $\pi\star\rho$
in left weak Bruhat order and their lengths differ by $k\times (n-k)$), 
and $P(\mathbf{i})$ the reduced word tableau of $\mathbf{i}$.

 Since $\pi$ and $\rho$ have separated descents,
the construction of $D\star E$ is for any
$D\in\BPD(\pi)$ and $E\in\BPD(\rho)$ is
possible. Suppose
$\rect(D\star E)=F$.
Since all marked blank tiles in $D\star E$
are in rows $[1-k,0]$, there are no
unmarked blank tiles in rows $[1-k,0]$, and rectification
of $D\star E$ removes all the marked blank tiles
while preserving the number of unmarked blank tiles
in each row, the pipes numbered $1-k$ to $0$ of
$F$ must all be for the identity permutation
on $[1-k,0]$. Therefore $F\in\BPD(\id_{[1-k,0]}\oplus\sigma)$ 
for some $\sigma\in S_{2n-k}$. That $\sigma\in\Sigma$
follows from the definition of rectification.
Then by Proposition~\ref{prop:undo}, \[C:=\{\Psi(D\star E, U):U\in\SYT(\lambda_{k\times(n-k)})\}\] is a single Coxeter-Knuth
class in reduced words of $(\pi\star\rho)\sigma^{-1}$,
and the corresponding reduced word tableaux
is of shape $\lambda_{k\times(n-k)}$.
So  $D\star E$ maps to $(F,C)$ under the bijection,
and $D\star E=F\leftarrow\mathbf{i}$ for any $\mathbf{i}\in C$.

For the other direction,
let $\sigma \in \Sigma$, and suppose
$\mathbf{i}$ is a reduced word of
$(\pi\star\rho)\sigma^{-1}$ such
that the shape of $P(\mathbf{i})$ is $\lambda_{k\times(n-k)}$.
Consider $\id_{[1-k,0]}\oplus\sigma\in S_{[1-k,2n-k]}$,
so the NW corner of any bumpless pipe dream of 
$\id_{[1-k,0]}\oplus\sigma$ is at $(1-k,1-k)$. 
Let $F\in\BPD(\id_{1-k,0}\oplus\sigma)$. 
By Lemma~\ref{lem:backstable} as well
as the length condition defining $\Sigma$,
 $F\leftarrow\mathbf{i}$
is defined, and therefore
the permutation of $(F\leftarrow\mathbf{i})$
is $\pi\star\rho$, by definition of $\mathbf{i}$ and
insertion.
By Proposition~\ref{prop:lls},
since the shape of the reduced word tableau
of $\mathbf{i}$ is $\lambda_{k\times(n-k)}$, the first
$k$ rows of $(F\leftarrow\mathbf{i})$ must
be of the form that the first $n-k$ columns 
consist of marked blank tiles, and the
rest of the columns contain no blank tiles.
This means
$(F\leftarrow\mathbf{i})\in\BPD^\boxtimes(\pi\star\rho)$.
If $\mathbf{j}$ is another reduced word of 
$(\pi\star\rho)\sigma^{-1}$ but in a different
Coxeter-Knuth class than $\mathbf{i}$,then
$(F\leftarrow\mathbf{i})\neq (F\leftarrow\mathbf{j})$
by Proposition \ref{prop:lls}. Therefore,
each pair $(F,C(\mathbf{i}))$ corresponds to
a unique element in
$\BPD^\boxtimes(\pi\star\rho)$. The insertion algorithm 
preserves the number of unmarked blank tiles on
each row, so $F$ and $F\leftarrow \mathbf{i}$
have the same monomial weight.

Then since $\BPD^\boxtimes(\pi\star\rho)$ bijects
with $\BPD(\pi)\times\BPD(\rho)$ in a weight-preserving
way by Lemma~\ref{lem:bijection}, the desired bijection
is established. Furthermore, since
the Edelman-Greene coefficient $j^{(\pi\star\rho)\sigma^{-1}}_{\lambda_{k\times(n-k)}}$ counts the
number of elements in the set $\{C(\mathbf{i}): \mathbf{i}\in\red((\pi\star\rho)\sigma^{-1}), \operatorname{shape}(P(\mathbf{i}))=\lambda_{k\times(n-k)}\}$,
the statement $c_{\pi,\rho}^\sigma=j^{(\pi\star\rho)\sigma^{-1}}_{\lambda_{k\times(n-k)}}$ follows. The bijection 
accounts for all positive structure constants,
so $c_{\pi,\rho}^\sigma=0$ for all $\sigma\not\in\Sigma$. 
\end{proof}

\begin{example}
Let $Q=((1,3),(2,5),(4,6))\in \SYT(\lambda_{3\times 2})$. Then $\Psi(D\star E, Q)=s_1s_2s_{-1}s_5s_0s_3$,
where $D\star E$ is as in 
Figure~\ref{fig:DstarE}.
The result
$\rect(D\star E)\in \BPD(\underline{-2}\,\underline{-1}03254167)$ is shown in
Figure~\ref{fig:rectres}.  $\S_{32541}$ appears with
multiplicity 1 in the Schubert product expansion of $\S_{13542}\S_{2143}$.
\begin{figure}[h]
    \centering
    \includegraphics[scale=0.55]{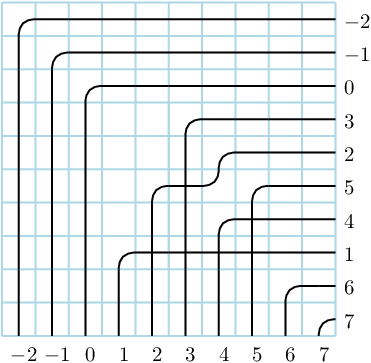}
    \caption{$\rect(D\star E)$}
    \label{fig:rectres}
\end{figure}
\end{example}
\begin{remark}Notice that we specifically imposed a length
condition relating $\pi,\rho,\sigma$ when
stating our rule. To show that this is 
necessary, consider the example when
$\pi=13542$, $\rho=21435$,
$\pi\star\rho=(-2,0,2,4,3,6,1,-1,5,7)$,
and $\sigma=5243167$. 
Then $(\pi\star\rho)\sigma^{-1}=(-2,0,2,-1,3,1,6,4,5,7)$, and $\mathbf{i}=s_1s_2s_{-1}s_5s_0s_4$ is a reduced word of $(\pi\star\rho)\sigma^{-1}$
whose reduced word tableau is of shape $3\times 2$.
Here $\ell(\sigma)=8$, $\ell(\pi\star\rho)=12$, and the required length condition is not satisfied.
Clearly, $\S_\sigma$ does not appear in the
expansion of $\S_\pi\S_\rho$ for degree reasons. Notice that the insertion of $\mathbf{i}$ into any $F\in\BPD(\sigma)$ is
not defined.
\end{remark}

As a direct consequence of
the proof of Theorem~\ref{thm:main}, we may state a different rule for the 
separated-descent structure constants that is
more closely analogous to one of the classical
rules for Littlewood-Richardson coefficients
using semi-standard Young tableaux, see e.g.\!
Corollary 2(v) of
\cite[Section 5.1]{fulton1997young}, without
making a connection to the Edelman-Greene
story. In this statement, we do not need to 
specify a length condition as in Theorem~\ref{thm:main}, as the bumpless pipe
dreams that arise as results of rectification are automatically for the correct permutations.
\begin{theorem}
\label{thm:jdtversion}
Let $\pi,\rho\in S_n$ be permutations with
separated descents at position $k$ and
$\sigma$ a permutation. 
Let $F\in\BPD(\id_{[1-k,0]\oplus\sigma})$.
Then $c_{\pi,\rho}^\sigma$ is the number of
elements $P$ in $\BPD^\boxtimes(\pi\star\rho)$ such
that $\rect(P)=F$.
\end{theorem}

\subsection{Multiple permutations with separated descents}
Equipped with the techniques developed
in the previous sections, we can easily
generalize the result to Schubert 
products for multiple permutations
with separated descents. To be precise,
suppose $\pi_0, \pi_1,\cdots \pi_m\in S_n$ and $1\le k_1 < k_2<\cdots <k_m\le n-1$,
where $\Des(\pi_0)\subseteq [1,k_1]$,
$\Des(\pi_1)\subseteq [k_1, k_2]$, $\cdots$ $\Des(\pi_i)\subseteq [k_i, k_{i+1}]$ $\cdots$,
$\Des(\pi_m)\subseteq [k_m,n-1]$.
The star operation can be generalized 
to the $(m+1)$-ary version, as follows:
\[\star(\pi_m,\cdots,\pi_1,\pi_0)(i) =\begin{cases}
\pi_m(i+\sum_{\alpha=1}^m k_\alpha)-(\sum_{\alpha=1}^m k_\alpha) & \text{ if }i\in [1-\sum_{\alpha=1}^m k_\alpha,-\sum_{\alpha=1}^{m-1} k_\alpha]\\
\pi_{m-1}(i+\sum_{\alpha=1}^{m-1} k_\alpha)-(\sum_{\alpha=1}^{m} k_\alpha)+n & \text{ if }i\in [1-\sum_{\alpha=1}^{m-1} k_\alpha,-\sum_{\alpha=1}^{m-2} k_\alpha]\\
\vdots & \\
\pi_1(i+k_1)-(\sum_{\alpha=1}^m k_\alpha)+(m-1)n & \text{ if }i\in [1-k_1,0] \\
\pi_0(i)+mn-\sum_{\alpha=1}^m k_\alpha & \text{ if }i\in [1,k_1] \\
\pi_1(i)+(m-1)n-\sum_{\alpha=1}^m k_\alpha & \text{ if }i\in [k_1+1,k_2] \\
\vdots & \\
\pi_m(i)-\sum_{\alpha=1}^m k_\alpha &
\text{ if }i\in [k_m+1, n]\\
\phi(i) &
\text{ if }i\in[n+1,(m+1)n-\sum_{\alpha=1}^m k_\alpha].
\end{cases}\]
where $\phi(i)$ is the $(i-n)$th smallest number in
$[1-\sum_{\alpha=1}^m k_\alpha,(m+1)n-\sum_{\alpha=1}^m k_\alpha]\setminus \pi\star\rho([1-\sum_{\alpha=1}^m k_\alpha, n])$. Here $\star(\pi_m,\cdots,\pi_0)\in S_{[1-\sum_{\alpha=1}^m k_\alpha, (m+1)n-\sum_{\alpha=1}^m k_\alpha]}$.

Define also 
$\BPD^\boxtimes(\star(\pi_m,\cdots \pi_0))$ as the set of
$D\in\BPD(\star(\pi_m,\cdots \pi_0))$
such that rows
$[1-\sum_{\alpha=1}^m k_\alpha,0]$ 
are of the form as showin in Figure~\ref{fig:firstrowsmulti}. We
denote the shape of the partition
at the NW corner as
$\lambda(k_1,\cdots,k_m,n)$.
\begin{figure}
    \centering
    \includegraphics[scale=0.8]{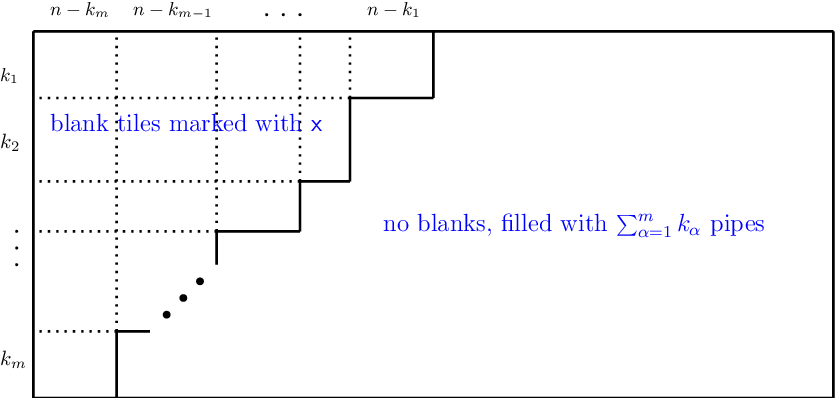}
    \caption{First $\sum_{\alpha=1}^m k_\alpha$ rows of $D\in \BPD^\boxtimes(\star(\pi_m,\cdots \pi_0))$}
    \label{fig:firstrowsmulti}
\end{figure}

We can easily generalize
the procedure described in Section \ref{subsec:sep1} and define the
$\star$-operation on $D_m\in \BPD(\pi_m),\cdots, D_0\in \BPD(\pi_0)$. The resulting bumpless
pipe dream is an element of
$\BPD^\boxtimes(\star(\pi_m,\cdots \pi_0))$. Furthermore, 
$\BPD^\boxtimes(\star(\pi_m,\cdots \pi_0))$ is in direct bijection
with $\BPD(\pi_m)\times \cdots\times \BPD(\pi_0)$.
See Figure~\ref{fig:multi}

\begin{figure}
    \centering
    \includegraphics[scale=0.6]{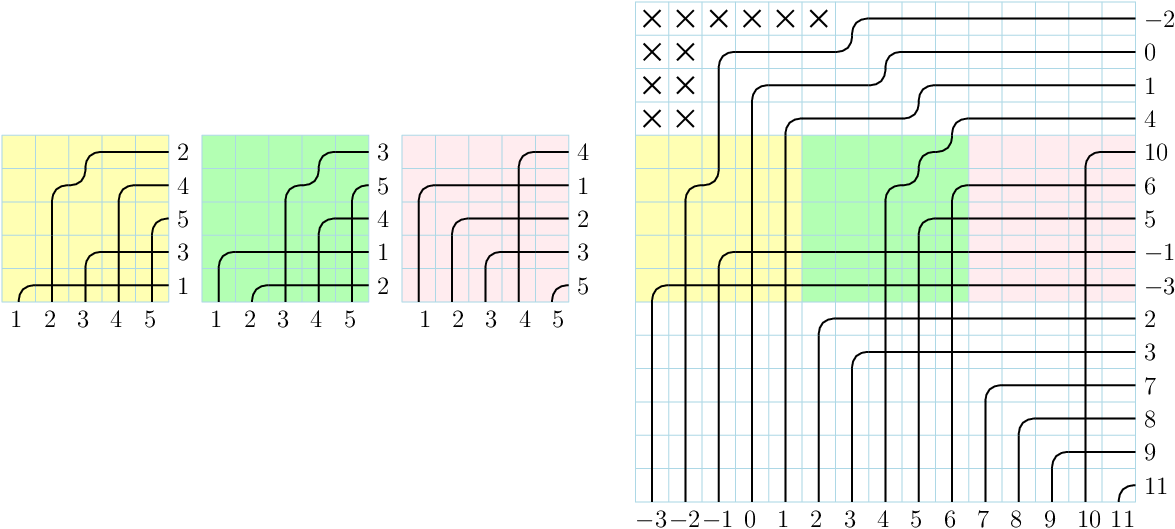}
    \caption{$n=5$, $m=2$, $k_1=1$, $k_2=3$, $\pi_0=41235$,
    $\pi_1=35412$, $\pi_2=24531$, $\star(\pi_2,\pi_1,\pi_0)\in[-3,11]$}
    \label{fig:multi}
\end{figure}
Our main theorem generalizes to 
Schubert products
for multiple permutations with
separated descents.

\begin{theorem}
Suppose $\pi_0, \pi_1,\cdots \pi_m\in S_n$ and $1\le k_1 < k_2<\cdots <k_m\le n-1$,
where $\Des(\pi_0)\subseteq [1,k_1]$,
$\Des(\pi_1)\subseteq [k_1, k_2]$, $\cdots$,$\Des(\pi_i)\subseteq [k_i, k_{i+1}]$,$\cdots$,
$\Des(\pi_m)\subseteq [k_m,n-1]$.
Let $\sigma\in S_{(m+1)n-\sum_{\alpha=1}^m k_\alpha}$. 
such that $\ell(\star(\pi_m,\cdots, \pi_0))-\ell(\sigma)=\ell((\star(\pi_m,\cdots, \pi_0))\sigma^{-1})=|\lambda(k_1,\cdots,k_m,n)|$. The Schubert structure
constant $c_{\pi_m,\cdots,\pi_0}^\sigma$ is equal to the 
Edelman-Greene coefficient $j_{\lambda(k_1,\cdots, k_m,n)}^{(\star(\pi_m,\cdots, \pi_0))\sigma^{-1}}$,
which is the number of
reduced word tableaux $T$
of shape $\lambda(k_1,\cdots, k_m,n)$
such that the permutation given by 
the reading word of $T$ is $(\star(\pi_m,\cdots, \pi_0))\sigma^{-1}$. Furthermore, $c_{\pi_m,\cdots,\pi_0}^\sigma=0$ for
all other $\sigma$.
\end{theorem}
\begin{remark}
Given $\pi_0,\pi_1,\pi_2$ that 
satisfy the separated descent conditions
as stated in the theorem, it is not
difficult to check that
for any $D_0\in \BPD(\pi_0)$, $D_1\in\BPD(\pi_1)$, $D_2\in \BPD(\pi_2)$,
$\rect(\star(D_0,D_1,D_2))=\rect(D_0\star(\rect(D_1\star D_2)))=\rect(\rect(D_0\star D_1)\star D_2)$.
\end{remark}

\section{Connection to jeu de taquin}
\label{sec:jdtconnection}
The construction of the star operation and
the rectification process were inspired by
the celebrated jeu de taquin algorithm
on skew tableaux. We discuss the connection
in this section.

Let $\lambda=(\lambda_1\ge \lambda_2\ge \cdots \ge \lambda_k\ge 0)$ be a partition. Let $w(\lambda,k)$ be the $k$-Grassmannian
permutation associated to partition $\lambda$,
namely, $w(\lambda,k)(i)=\lambda_{k-i+1}+i$ for $i\le k$
and $w(\lambda, k)$ has no descents after $k$.
For any $D\in\BPD(w(\lambda,k))$, associate 
a semi-standard Young tableau $T_D$ to
$D$ as follows. Label each blank tile
in $D$ with its row index, and perform 
all possible droop moves in $D$.
Each droop moves a blank tile
towards the NW direction together with
its label. When finished, the semi-standard
Young tableau at the NW corner is $T_D$.
Denote the shape of $T_D$ as $\lambda$.
It is known from \cite{LLS}
that this map $D\mapsto T_D$ gives a
bijection between $\BPD(w(\lambda,k))$ and
$\SSYT_k(\lambda)$,
where $\SSYT_k(\lambda)$ denotes
the set of semi-standard Young tableaux
of shape $\lambda$ with entries no greater
than $k$. The relationship between bumpless pipe dreams and tableaux in some sense dates back to \cite{KMY}, before the bumpless pipe dreams formulation was explicitly stated.

Let $\nu=(\nu_1,\cdots,\nu_m)$ and 
$\mu=(\mu_1,\cdots,\mu_l)$ be partitions
such that $l<m$ and $\mu_i\le\nu_i$
for each $i$. Suppose $k\ge m$. 
Let $T\in \SSYT_k(\nu/ \mu)$. 
We complete $T$ into a semi-standard
tableau of shape $\nu$
by filling in for each row
$i=l,l-1,\cdots, 1$, the number $i-l$
in each box on row $i$ that belongs
to $\mu$. We also mark these
entries with a different color. 
Call the resulting marked semi-standard
tableau $T'$. We may now use the correspondence
between semi-standard Young tableaux of
straight shapes and
bumpless pipe dreams for Grassmannian permutations
to get a bumpless pipe dream $D_{T'}$ for $T'$. To
distinguish the non-positive entries from the
original entries, we mark each of the blank tiles
in the non-positive rows of $D_{T'}$
with an ``$\mathsf{x}$''.

\begin{figure}[h]
    \centering
    \includegraphics[scale=0.72]{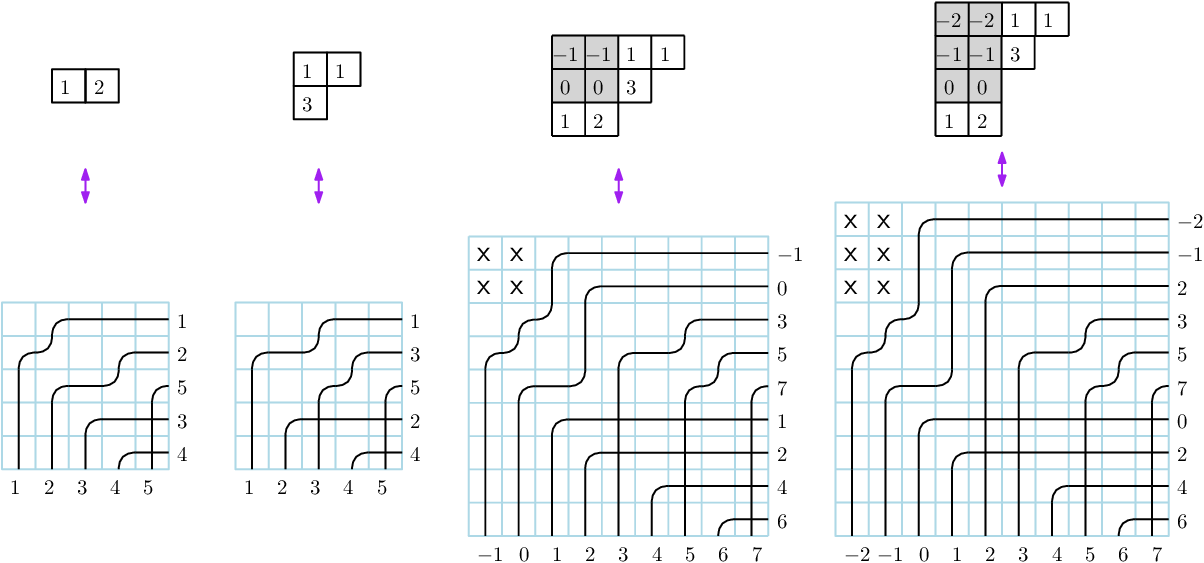}
    \caption{Connection of $T_D*T_E$ with $D\star E$}
    \label{fig:grassstar}
\end{figure}
We will now demonstrate by an example how our construction
connects to the classical construction
for computing products of Schur polynomials
with semi-standard Young tableaux, see Figure~\ref{fig:grassstar}. In this example,
$k=3$ and $n=5$. We let $T_D$ denote the 
first tableau, $D\in\BPD(12534)$ its corresponding
bumpless pipe dream, $T_E$ the second tableau, and
$E\in\BPD(13524)$. $T_D*T_E$ is the skew tableau
of shape $(4,3,2)/(2,2)$
formed by placing $T_E$ immediately NE of
$T_D$. We fill in the boxes in the NW partition
of shape $(2,2)$ with non-positive numbers as
shown, and find its corresponding bumpless
pipe dream $F$ with marked blank tiles. 
Finally, we construct $D\star E$ and find
its corresponding skew tableau. Notice
that for both $D\star E$ and $F$, drooping and
undrooping within positive rows will generate
a set in direct bijection with 
$\BPD(12534)\times \BPD(13524)$. Therefore, our
$\star$-construction is not the unique choice,
but a convenient one.

% The following picture shows an example
% of the map. 
% \begin{figure}
%     \centering
%     \includegraphics[scale=0.8]{tab3.eps}
%     \caption{Correspondence
%     between $\BPD(\lambda)$ and 
%     $\SSYT_k(\lambda)$}
%     \label{fig:gr}
% \end{figure}

It remains to explain how jeu de taquin
relates to our column moves.
Let $T'$ be a marked semi-standard Young
tableau, as described prior to the example.
Define a
rectification algorithm on $T'$ by
iterating the following procedure
until there are no more marked entries.

    \begin{enumerate}
        \item Pick a SE most marked
        number in the NW marked region
        to be ``active'' in this iteration.
        \item 
        Increment the active marked number by 1. If this number is equal to $k$ and 
        at the SE border, delete it and end the iteration. Otherwise:
        \begin{enumerate}
            \item If the active marked
            number is equal to the number below, increment
            the number below as well
            and move the mark there.
            Repeat this until we get
            a valid tableau again.
            \item If the active marked
            number is equal to the 
            number to the right, move
            the mark to the rightmost
            number with this value on
            this row.
        \end{enumerate}
        Repeat this step.
    \end{enumerate}

\begin{proposition}
The rectification algorithm
    on $T'$ described above simulates
    jeu de taquin on $T$.
    Furthermore, the
    algorithm agrees with the
    rectification algorithm defined
    on marked bumpless pipe dreams
    in each step when each 
    intermediate marked tableau
    is mapped 
    to its corresponding marked 
    bumpless pipe dream.
\end{proposition}
\begin{proof}
Step 2(a) simulates consecutive
vertical slides in jeu de taquin,
and Step 2(b) simulates 
consecutive horizontal slides.
Furthermore, Step 2(a) corresponds
to an undroop column move
as defined in Section \ref{subsec:recinsert}, and 
Step 2(b) corresponds to  moving
a mark to the rightmost tile
in a consecutive block of blank 
tiles on a row in a bumpless pipe
dream. Finally, in a 
bumpless pipe dream of a $k$-Grassmannian permutation,
the column move that removes
a cross is only available
when the marked tile is in row
$k$. This determines the 
terminating condition in Step 2.
\end{proof}

\section{Final remarks}
It was shown in \cite{LLS}
that bumpless pipe dreams 
compute double Schubert polynomials,
and further explained in \cite{weigandt2021bumpless} that 
they also compute double 
Grothendieck polynomials.
We note that our construction
cannot be easily extended to
an equivariant setting, because 
the $\star$-operation and the insertion/rectification
algorithm are
row-biased and do not maintain
the necessary invariance for columns. 
A similar phenomenon is exhibited
in \cite{knutson2004formula}.
However, it should be possible to 
generalize our result to $K$-theory. 
A $K$-theoretic version of Edelman-Greene
insertion (perhaps Hecke insertion \cite{buch2008stable}) for bumpless pipe dreams,
would be necessary. We leave this for
future work.
An interpretation of the
$K$-theoretic version of
jeu de taquin \cite{thomas2009jeu} in the Schubert setting might also be relevant.
It would also be interesting to
see what other Schubert structure constants, if any,
are also Edelman-Greene coefficients.

\vskip 1em
\noindent\textbf{Funding.}
This material is based upon work supported by the National Science Foundation under Grant No. DMS-1439786 while the author was in residence at the Institute for Computational and Experimental Research in Mathematics in Providence, RI, during the Combinatorial
Algebraic Geometry program.
\vskip 1em

\noindent\textbf{Acknowledgments.}
I am grateful to my advisor Allen Knutson
for many in-depth discussions and inspiring questions,
as well as extensive comments on the paper. 
I would also like to thank Yibo Gao,
Thomas Lam, Mark Shimozono, and
Alex Yong for helpful discussions, as well as the anonymous
referees for careful reading and helpful suggestions. 
The improved exposition of the $\star$-operation on bumpless pipe dreams is due to one of them.

\bibliographystyle{alpha}
\bibliography{ref}

\end{document}